\documentclass[a4paper]{article}
\usepackage{a4wide}
\usepackage[utf8]{inputenc}
\usepackage[english]{babel}
\usepackage{amssymb}
\usepackage{amsfonts}
\usepackage{amsmath}
\usepackage{amsthm}
\usepackage{graphicx}
\usepackage{hyperref}

\newtheorem{lemma}{Lemma}[section]
\theoremstyle{remark}
\newtheorem{remark}{Remark}[section]
\theoremstyle{definition}
\newtheorem{definition}{Definition}[section]

\DeclareMathOperator{\Span}{span}
\newcommand{\Prob}{\mathsf{P}}
\newcommand{\Expect}{\mathsf{E}}

\begin{document}

\begin{center}
\LARGE{Branching Random Walks with Two Types of Particles on Multidimensional Lattices}
\end{center}
\begin{center}
\textit{Iu.\,Makarova$^1$, D.\,Balashova$^1$, S.\,Molchanov$^{2}$, E.\,Yarovaya$^1$}
\end{center}

\begin{center}
$^1$ Department of Probability Theory, Lomonosov Moscow State University,\\
$^2$ National Research University Higher School of Economics,\\
Moscow, Russia\\
\end{center}
\medskip

\begin{abstract}
We consider a continuous-time branching random walk on a multidimensional lattice with two types of particles and an infinite number of initial particles. The main results are devoted to the study of the generating function and the limiting behavior of the moments of subpopulations generated by a single particle of each type. We assume that particle types differ from each other not only by the laws of branching, as in multi-type branching processes, but also by the laws of walking. For a critical branching process at each lattice point and recurrent random walk of particles, the effect of limit spatial clustering of particles over the lattice is studied. A model illustrating epidemic propagation is also considered. In this model, we consider two types of particles: infected and immunity generated. Initially, there is an infected particle that can infect others. Here, for the local number of particles of each type at a lattice point, we study the moments and their limiting behavior. Additionally, the effect of intermittency of the infected particles is studied for a supercritical branching process at each lattice point. Simulations are presented to demonstrate the effect of limit clustering for the epidemiological model.

\medskip

\end{abstract}

\section{Introduction} \label{introduction}
The branching random walk (BRW) is one of the widely used tools for describing the processes associated with the birth, death, migration, and immigration of particles~\mbox{\cite{YarBRW:e,evlbul,BLP15,BN18}}. BRWs occur in population dynamics~\cite{MHMY} and have numerous applications, e.g., in genetics~\cite{MKY:21} and demography~\cite{MW2017}.

The continuous-time BRWs presented in this paper are the two-type branching processes with walking of particles which takes place on the multidimensional lattice $\mathbb Z^d$, $d\in \mathbb N$. We will mainly study the distribution of subpopulations generated by a single particle of each type. We assume that each particle can produce not only particles of the same type, but also particles of a different type. We also assume that particles cannot change their type. However, later in Section~\ref{important_example}, we lift this condition for a particular case.

The problem of multi-type processes is one of the most interesting and really complicated problems in the theory of random processes. Historically, these processes were apparently first considered by Sevastyanov in~\cite{Sevast}. He considered both discrete and continuous-time branching processes with a finite number of types and studied the limit distribution of particles under different conditions. Nowadays this problem is also studied in detail by various research groups. In~\cite{BH1,BH2}, for example, the authors consider a more complicated problem in which the number of types is not finite but countable. They consider the subclass of Galton--Watson processes called lower Hessenberg branching processes. In contrast to our studies, they investigated processes with discrete time, and the random walk was considered in a strip. Some works of Vatutin with co-authors~\cite{VW,VD} are devoted to the problem of multi-type branching processes with discrete time and finite number of types in the random environment, but without walking of particles.

The structure of the paper is as follows. In Section~\ref{description_of_the_model}, we describe a two-type BRW on $\mathbb Z^{d}$ with infinitely many initial particles of both types. Here, we also define the main objects of the study. In Section~\ref{the_first_moments}, we study the first moments for subpopulations generated by a single particle of each type, and find their asymptotic behavior at the sites of $\mathbb Z^{d}$. To this end, we first obtain differential equations for the generating functions of subpopulations generated by a single particle of each type in Lemma~\ref{lemma_first_moments} from Section~\ref{E:DUM1}. In order to find solutions for the corresponding equations, we turn to the equations for the Fourier transforms of the corresponding moments in Section~\ref{solution_first_moments} and show that the corresponding equations can be solved explicitly. This allows us to obtain explicit solutions~\eqref{m_ij_b=0_c>0} and \eqref{m_ij_b>0_c>0} for the Fourier transforms of the corresponding moments. The results are then applied in Section~\ref{E:AsympFVJ} to find the asymptotics of the solutions for the Fourier transforms of the first moments of the subpopulations in the case of finite variance of the jumps. In Section~\ref{the_second_moments}, we study the second moments for the subpopulations. In Section~\ref{Cluster}, we study the particle clustering effect for BRWs under additional assumptions. Here, we assume that a two-type branching process is critical at every point on $\mathbb Z^{d}$ and the distribution of an underlying random walk jumps has light tails. In Section~\ref{important_example}, we discard one of the assumptions we imposed on our model in Sections~\ref{description_of_the_model}--\ref{the_second_moments} and instead assume that the particles of the first type can occasionally change their type to the second. This model can illustrate the situation related to epidemic spread, especially the spread of COVID-19 around the world. We refer to the first type of particles as infected and the second type of particles as immunity generated against COVID-19. Here, we assume that infected particles change their type after a short period of time so that they build up immunity. However, we assume that only one particle can do this after a short period of time. In Section~\ref{simulation_of_brw}, the algorithm for modeling the processes studied in~Sections~\ref{Cluster} and~\ref{important_example} is presented and examined using the Python programming language.

\section{Description of the Model}
\label{description_of_the_model}

Here, we consider a population model with two types of particles. Let $N_i(t,y)$, with $i=1,2$, be the number of particles of type $i$ at time $t > 0$ at the site $y\in\mathbb Z^d$, $d\geqslant1$. Then, the total population at the point $y\in\mathbb Z^d$ at time $t > 0$ can be represented as the following column-vector whose components are non-negative integers:
\begin{equation}\label{total_population}
N(t,y) = [N_1(t,y),N_2(t,y)]^{T}.
\end{equation}

We assume that $N_i(0,x) = l_i$ for $i=1,2$ and all $x\in\mathbb Z^d$.

We assume that the evolution of particles of each type consists of several possibilities. First, a particle of type $i$, $i=1,2$, can die with mortality rate $\mu_i \geqslant 0$. Second, each particle of type $i$ can produce new particles of either type. We denote by $\beta_i (k,l) \geqslant 0$, $k+l\geqslant2$, the rate at which a particle of type $i$ produces $k$ particles of type $i=1$ and $l$ particles of type $i=2$. Then, we define the corresponding branching generation function (without particle death) for $i=1,2$, see, e.g.,~\cite{Sevast}:
\begin{equation}\label{gen_function}
F_i(z_1,z_2) = \sum_{k+l\geqslant2} z_1^k z_2^l \beta_i (k,l).
\end{equation}

\begin{remark}\label{remark_mu_beta}
In our notation, $\mu_i = \beta_i (0,0)$ ($i=1,2$), and we do not consider the case when a particle of the type $i=1,2$ can transform to a particle of type $j=1,2$, $j\ne i$, and hence $\beta_1 (0,1) = \beta_2 (1,0) = 0$. Additionally, we assume that
\begin{align*}
\mu_1 + \sum_{k+l\geqslant2} \beta_1 (k,l) &= -\beta_1 (1,0) > 0,\\
\mu_2 + \sum_{k+l\geqslant2} \beta_2 (k,l) &= -\beta_2 (0,1) > 0,
\end{align*}
where $\beta_1 (1,0)$ and $\beta_2 (0,1)$ denote the cases when nothing happens with the particles.
\end{remark}

Remember that particles can jump between points on the lattice. We assume that the probability of a jump from a point $x$ to a point $x+v$ during the small period $dt$ is equal to $\varkappa_i a_i(x,x+v)\,dt + o(dt)$, $i=1,2$. Here, $\varkappa_i > 0$ is the diffusion coefficient. In what follows, we consider a symmetric random walk, i.e., the case where $a_i (x,y) = a_i (y,x)$. Moreover, we assume that the random walk is homogeneous in space: $a_i (x,x+v) = a_i (v)$ and irreducible such that $\Span\{ v: a_i (v) > 0 \} = \mathbb Z^d$. Moreover, $a_i (0) = -1$, $\sum_{v} a_i (v) = 0$.

Then, the migration operator has the form
\begin{equation}
\label{generator}
(\mathcal L_i \psi)(x) := (\mathcal L_i \psi(\cdot))(x) := \varkappa_i \sum_v (\psi(x+v)-\psi(x))a_i(v).
\end{equation}

Let us introduce the subpopulations, which can be represented as the following column-vectors:
\begin{equation}\label{subpopulations}
\begin{aligned}
n_1 (t,x,y) &= [n_{11} (t,x,y),n_{12} (t,x,y)]^{T},\\
n_2 (t,x,y) &= [n_{21} (t,x,y),n_{22} (t,x,y)]^{T}.
\end{aligned}
\end{equation}

Here, $n_i (t,x,y)$ is the vector of particles at the point $y$, generated by a single particle of type $i$ which at time moment $t=0$ was at the site $x\in\mathbb Z^d$. Its components $n_{ij} (t,x,y)$ are the numbers of particles at the point $y$ of type $j$, generated by a single particle of type $i$ at $x$ at the moment $t=0$. Note that
\begin{equation}\label{initial_conditions}
n_{ij} (0,x,y) = \delta_i (j) \delta_x (y),
\end{equation}
where $\delta_u (v)$ is the Kronecker function on $\mathbb Z^d$ (or $\mathbb R$), that is if $u$, $v\in\mathbb Z^d$ (or $\mathbb R$)
\[
\delta_u (v) = \begin{cases}
1,\quad u = v;\\
0,\quad u \ne v.
\end{cases}
\]

\begin{remark}\label{R-uniform}
In the BRW under consideration we assume that both random walk and branching process are ``homogeneous''. Namely, we assume that underlying random walk for each type of particles $i=1,2$ is homogeneous in space, so that $a_i (x,y)=a_i (x-y,0) = a_i (x-y)$. At the same time branching process (which includes death and birth of particles) is also ``homogeneous'' due to the fact that all intensities $\mu_i$, $\beta_i (k,l)$, $k+l\geqslant2$, $i=1,2$ are constant and only depend on the type of particles (and independent of the lattice points).

Such a ``homogeneity'' leads to the simplifying the relations which describe the evolution of considered BRW. First of all, we conclude that for all $t\geqslant0$ the probability $\Prob\big(n_{ij} (t,x,y)=k \big)$ equals to $\Prob\big(n_{ij} (t,x-y,0)=k \big)$ for all $k\in\mathbb Z_+$, so that
\begin{equation}\label{E-nij}
\Prob\big(n_{ij} (t,x,y)=k \big)\equiv\Prob\big(n_{ij} (t,x-y,0)=k \big),\qquad t\geqslant0.
\end{equation}
To prove this equality, we consider the process $n_{ij} (t,x,y)$, which starts at some lattice point $x$, so that $n_{ij} (0,x,y)=\delta_i(j)\delta_x(y)$. Then, for each trajectory
\[
x\mapsto x_{1}\mapsto x_{2}\mapsto\cdots\mapsto x_{n-1}\mapsto y
\]
which describes the transition of a particle from point $x$ to the point $y$, there exists the ``trajectory with a shift of $y$''
\[
x-y\mapsto x_{1}-y\mapsto x_{2}-y\mapsto\cdots\mapsto x_{n-1}-y\mapsto 0
\]
which describes the transition of a particle from point $x-y$ to the point $0$. At the same time, due to the homogeneity in space of random walk, all transition intensities for both trajectories are equal ($a_i (x,y)=a_i (x-y,0)=a_i(x-y)$), and because of the ``branching homogeneity'', all branching intensities are equal at every lattice point. As $n_{ij} (0,x-y,0) = n_{ij} (0,x,y)$, then for all $t\geqslant0$ we can conclude that Equation~\eqref{E-nij} is true.

From~\eqref{E-nij}, we get that for all values which can be obtained from $n_{ij} (t,x,y)$ we have the same relations. In particular, for $\Expect n_{ij} (t,x,y)$ we get
\begin{equation}\label{E-mij}
\Expect n_{ij} (t,x,y)\equiv \Expect n_{ij} (t,x-y,0),\qquad t\geqslant0.
\end{equation}

Finally, note that the similar relation~\cite{MolchanovYarovaya2013} holds for transition probabilities $p_i(t,x,y)$, $i=1,2$ (definition will be given later)
\begin{equation}\label{E-pij}
p_{i} (t,x,y)\equiv p_{i} (t,x-y,0),\qquad t\geqslant0.
\end{equation}
The proof of the previous relation can be also obtained from the representation~\eqref{E-pi} from Section~\ref{E:AsympFVJ}.
\end{remark}

\begin{remark}\label{R-simplification}
From Equations~\eqref{E-nij} and \eqref{E-pij} obtained in Remark~\ref{R-uniform} we conclude that to investigate considered BRW, which starts at the lattice point $x$ it is sufficient to consider the case $x=0$. That can simplify the future narration.
\end{remark}

Now, using notation from Equation~\eqref{subpopulations}, we obtain the following representation of the total population specified by Equation~\eqref{total_population}:
\begin{equation}\label{N(t,y)}
 N (t,y) = \sum_{x\in\mathbb Z^d}\sum_{s\in\{ 1,\ldots,l_1 \}} n_{1,s} (t,x,y) + \sum_{x\in\mathbb Z^d} \sum_{m\in\{ 1,\ldots,l_2 \}} n_{2,m} (t,x,y)\quad,
\end{equation}
where $n_{i,l} (t,x,y)$ is the subpopulation generated by the $l$-th particle at the point $x$ at the time $t=0$. Note that both internal series in Equation~\eqref{N(t,y)} do not depend on the order of enumeration of particles.

The components of the vector $N(t,y)$ are
\begin{equation}\label{N_i(t,y)}
 N_i (t,y) = \sum_{x\in\mathbb Z^d}\sum_{s\in\{ 1,\ldots,l_1 \}} n_{1i,s} (t,x,y) + \sum_{x\in\mathbb Z^d}\sum_{m\in\{ 1,\ldots,l_2 \}} n_{2i,m} (t,x,y),
\end{equation}where $i=1,2$

Given $z = (z_1,z_2)$, let us introduce the generating function
\begin{equation}\label{subpop_gen_function}
\Phi_i (t,x,y;z) = \Expect z_1^{n_{i1} (t,x,y)} z_2^{n_{i2} (t,x,y)}.
\end{equation}

This generating function specifies the evolution of a single particle of type $i=1,2$. Let us consider what can happen to this particle (later we can use it to obtain a differential equation for the generating functions). First, the initial particle can die at a point $x$ with probability $\mu_i \,dt + o(dt)$ (then the subpopulation of this particle disappears). Second, this particle can produce $k$ particles of type $1$ and $l$ particles of type $2$ with probability $ \beta_i (k,l)\,dt + o(dt)$. Third, the particle can jump from a point $x$ to a point $x+v$ with probability $ \varkappa_i a_i (v)\,dt + o(dt)$. Finally, nothing can happen to a particle during time $dt$. From this, we get
\begin{lemma}\label{lemma_gen_function}
The generating functions $\Phi_i (t,x,y;z)$, $i=1,2 $, specified by Equation~\eqref{subpop_gen_function}, satisfy the differential equation
\begin{align}\nonumber
\frac{\partial \Phi_i (t,x,y;z)}{\partial t} &= (\mathcal L_i \Phi_i (t,\cdot,y;z))(x) + \mu_i (1-\Phi_i (t,x,y;z))\\\label{E:L21-1}
&\quad + F_i (\Phi_1 (t,x,y;z),\Phi_2 (t,x,y;z))
- \sum_{k+l\geqslant2} \beta_i (k,l) \Phi_i (t,x,y;z);\\
\label{E:L21-2}
\Phi_i (0,x,y;z) &= \begin{cases}
1, &x\ne y;\\
z_i, &x=y.
\end{cases}
\end{align}
\end{lemma}

\begin{proof}
Given an $i=1,2$, consider the generating function $\Phi_i (t,x,y;z)$ at the time moment $t+dt$:
\begin{align*}
\Phi_i (t+dt,x,y;z) &= \Bigl(1-\varkappa_i \,dt -\mu_i \,dt - \sum_{k+l\geqslant2} \beta_i (k,l)\,dt\Bigr)\Phi_i (t,x,y;z) \\&\quad+ \varkappa_i \sum_v \Phi_i (t,x+v,y;z)a_i(v)\,dt + \mu_i \,dt\\&\quad+ \sum_{k+l\geqslant2} \beta_i (k,l) \Phi_1^k (t,x,y;z) \Phi_2^l (t,x,y;z)\,dt + o(dt).
\end{align*}

Then,
\begin{align*}
\Phi_i (t+dt,x,y;z)&-\Phi_i (t,x,y;z) = -\Bigl(\varkappa_i +\mu_i  + \sum_{k+l\geqslant2} \beta_i (k,l)\Bigr)\Phi_i (t,x,y;z)\,dt \\&\quad+ \varkappa_i  \sum_v \Phi_i (t,x+v,y;z)a_i(v)\,dt + \mu_i \,dt\\&\quad+ \sum_{k+l\geqslant2} \beta_i (k,l) \Phi_1^k (t,x,y;z) \Phi_2^l (t,x,y;z)\,dt + o(dt).
\end{align*}

Therefore,
\begin{align*}
\frac{\partial \Phi_i (t,x,y;z)}{\partial t} &= (\mathcal L_i \Phi_i (t,\cdot,y;z))(x) + \mu_i (1-\Phi_i (t,x,y;z)) \\&\quad+ \sum_{k+l\geqslant2} \beta_i (k,l) (\Phi_1^k (t,x,y;z)\Phi_2^l (t,x,y;z)-\Phi_i (t,x,y;z)).
\end{align*}

Here, according to Equation~\eqref{gen_function}, we have
\[
\sum_{k+l\geqslant2} \beta_i (k,l) \Phi_1^k (t,x,y;z)\Phi_2^l (t,x,y;z) = F_i (\Phi_1 (t,x,y;z),\Phi_2 (t,x,y;z)),
\]
and hence
\begin{align*}
\frac{\partial \Phi_i (t,x,y;z)}{\partial t} &= (\mathcal L_i \Phi_i (t,\cdot,y;z))(x) + \mu_i (1-\Phi_i (t,x,y;z)) \\&\quad+ F_i(\Phi_1 (t,x,y;z), \Phi_2 (t,x,y;z)) - \sum_{k+l\geqslant2} \beta_i (k,l)\Phi_i (t,x,y;z)).
\end{align*}

The initial condition for the latter equation follows from Equation~\eqref{initial_conditions}:
\begin{align*}
\Phi_i (0,x,y;z) &= \Expect z_1^{n_{i1} (0,x,y)} z_2^{n_{i2} (0,x,y)} = \Expect z_1^{\delta_i (1) \delta_x (y)} z_2^{\delta_i (2) \delta_x (y)} \\&= z_1^{\delta_i (1) \delta_x (y)} z_2^{\delta_i (2) \delta_x (y)} = z_i^{\delta_x (y)}.
\end{align*}

So, we obtain the desired Equations~\eqref{E:L21-1} and~\eqref{E:L21-2}, which completes the proof of Lemma~\ref{lemma_gen_function}.
\end{proof}

\begin{remark}\label{rem_carleman}
If we assume
\begin{equation}\label{int_ineq}
\beta_i (k,l) \leqslant \frac{c_0^{k+l}}{k!l!},\quad k+l\geqslant2,
\end{equation}
for some $c_0 > 0$, then the Carleman condition is hold~\cite{YSK:2020}, which guarantees that for each $i=1,2$ the function $F_i (z_1,z_2)$ from Equation~\eqref{gen_function} is an analytic function in the strip $|z_i - 1| < \delta_0$~for some $\delta_0 >0$~\cite{Stoy:e}.
\end{remark}

\section{The First Moments}
\label{the_first_moments}
Recall that the goal of our article is to study the moments of the random variables $n_{ij} (t,x,y)$, $i,j=1,2$. In this section, we will consider the first moments. For this purpose, in Section~\ref{E:DUM1}, in Lemma~\ref{lemma_first_moments}, we obtain differential equations for the generating functions of subpopulations generated by a single particle of each type. In order to find solutions to the corresponding equations, we turn to the equations for the Fourier transforms of the corresponding moments in Section~\ref{solution_first_moments} and show that the corresponding equations can be solved explicitly. This allows us to obtain explicit solutions~\eqref{m_ij_b=0_c>0} and \eqref{m_ij_b>0_c>0} for the Fourier transforms of the corresponding moments. The results are then applied in Section~\ref{E:AsympFVJ} to find the asymptotics of the solutions for the Fourier transforms of the first moments of subpopulations in the case of finite variance of the jumps.

\subsection{Differential Equations for Moments}\label{E:DUM1}
Define $m_{ij}^{(1)} (t,x,y) = \Expect n_{ij} (t,x,y)$ and prove the following lemma playing important role in what follows.

\begin{lemma}\label{lemma_first_moments}
Let Equation~\eqref{int_ineq} be true. Then, for each $i,j=1,2$, the functions $m_{ij}^{(1)} (t,x,y)$ satisfy the differential equation
\begin{align}\notag
\frac{\partial m_{ij}^{(1)} (t,x,y)}{\partial t} &= (\mathcal L_i m_{ij}^{(1)} (t,\cdot,y))(x) - \mu_i m_{ij}^{(1)} (t,x,y) - \sum_{k+l\geqslant2} \beta_i (k,l) m_{ij}^{(1)} (t,x,y) \\
\label{El21-1}&\quad+ \sum_{k+l\geqslant2} \beta_i (k,l) (km_{1j}^{(1)} (t,x,y) + lm_{2j}^{(1)} (t,x,y));\\
\label{El21-initcind}m_{ij}^{(1)} (0,x,y) &= \delta_i (j) \delta_x (y).
\end{align}
\end{lemma}

\begin{proof}
Differentiating  Equation~\eqref{subpop_gen_function} with respect to $z_{j}$, $j=1,2$, we get
\[
\frac{\partial \Phi_i (t,x,y; z)}{\partial z_j} = \frac{\partial \Expect z_1^{n_{i1} (t,x,y)} z_2^{n_{i2} (t,x,y)}}{\partial z_j} = \Expect n_{ij} (t,x,y) z_1^{n_{i1} (t,x,y) - \delta_j (1)} z_2^{n_{i2} (t,x,y) - \delta_j (2)},
\]
from which, by taking $z=(z_{1},z_{2})=(1,1)$, we obtain
\begin{equation}\label{gen_function_first_moment}
\frac{\partial \Phi_i (t,x,y; z)}{\partial z_j}\Bigr|_{z=(1,1)} = \Expect n_{ij} (t,x,y) = m_{ij}^{(1)} (t,x,y).
\end{equation}

Now, differentiating Equation~\eqref{E:L21-1} over $z_j$ we can write
\begin{align*}
\frac{\partial^2 \Phi_i (t,x,y; z)}{\partial t \partial z_j} &= \partial_{z_j} \Bigl( (\mathcal L_i \Phi_i (t,\cdot,y;z))(x) + \mu_i (1-\Phi_i (t,x,y;z)) \\&\quad+ \sum_{k+l\geqslant2} \beta_i (k,l) (\Phi_1^k (t,x,y;z)\Phi_2^l (t,x,y;z)-\Phi_i (t,x,y;z)) \Bigr) \\&= \bigl(\mathcal L_i (\partial_{z_j} \Phi_i (t,\cdot,y;z))\bigr)(x) - \mu_i \bigl(\partial_{z_j} \Phi_i (t,x,y;z)\bigr) \\&\quad+ \sum_{k+l\geqslant2} \beta_i (k,l) \Bigl( k\bigl(\partial_{z_j} \Phi_1 (t,x,y;z)\bigr)\Phi_1^{k-1} (t,x,y;z)\Phi_2^l (t,x,y;z) \\&\quad+ l\Phi_1^k (t,x,y;z)\bigl(\partial_{z_j} \Phi_2 (t,x,y;z)\bigr)\Phi_2^{l-1} (t,x,y;z) - \bigl(\partial_{z_j} \Phi_i (t,x,y;z)\bigr) \Bigr).
\end{align*}

Again, by taking $z=(z_{1},z_{2})=(1,1)$ in the above formula and applying Equation~\eqref{gen_function_first_moment}, we find that the left-hand part of the last equation takes the form:
\begin{align}\label{lemma_left_part_equation}
\frac{\partial^2 \Phi_i (t,x,y; z)}{\partial t \partial z_j}\Bigr|_{z = (1,1)} = \frac{\partial m_{ij}^{(1)} (t,x,y)}{\partial t};
\end{align}
while the right-hand part of the same equation equals
\begin{align}\label{lemma_right_part_equation}
&\biggl(\bigl(\mathcal L_i (\partial_{z_j} \Phi_i (t,\cdot,y;z))\bigr)(x) - \mu_i \bigl(\partial_{z_j} \Phi_i (t,x,y;z)\bigr) + \sum_{k+l\geqslant2} \beta_i (k,l) \Bigl( k\bigl(\partial_{z_j} \Phi_1 (t,x,y;z)\bigr)\notag\\&\quad\times\Phi_1^{k-1} (t,x,y;z)\Phi_2^l (t,x,y;z) + l\Phi_1^k (t,x,y;z)\bigl(\partial_{z_j} \Phi_2 (t,x,y;z)\bigr)\notag\\&\quad\times\Phi_2^{l-1} (t,x,y;z) - \bigl(\partial_{z_j} \Phi_i (t,x,y;z)\bigr) \Bigr)\biggr) \Bigr|_{z=(1,1)} = (\mathcal L_i m_{ij}^{(1)} (t,\cdot,y))(x) \notag\\&\quad- \mu_i m_{ij}^{(1)} (t,x,y) + \sum_{k+l\geqslant2} \beta_i (k,l)\bigl(k-1) m_{1j}^{(1)} (t,x,y) + lm_{2j}^{(1)} (t,x,y).
\end{align}

By combining Equation~\eqref{lemma_left_part_equation} and \eqref{lemma_right_part_equation}, we obtain
\begin{align*}
\frac{\partial m_{ij}^{(1)} (t,x,y)}{\partial t} &= (\mathcal L_i m_{ij}^{(1)} (t,\cdot,y))(x)- \mu_i m_{ij}^{(1)} (t,x,y) - \sum_{k+l\geqslant2} \beta_i (k,l) m_{ij}^{(1)} (t,x,y) \\&\quad+ \sum_{k+l\geqslant2} \beta_i (k,l) (km_{1j}^{(1)} (t,x,y) + lm_{2j}^{(1)} (t,x,y)).
\end{align*}

The initial condition for the latter equation can be found from Equation~\eqref{initial_conditions}:
\[
m_{ij}^{(1)} (0,x,y) = \Expect n_{ij} (0,x,y) = \Expect \delta_i (j)\delta_x (y) = \delta_i (j)\delta_x (y).
\]

Lemma~\ref{lemma_first_moments} is proved.
\end{proof}

\begin{remark}\label{R:moments1}
From Lemma~\ref{lemma_first_moments} and the general theory of differential equations (in Banach spaces), for any $i,j=1,2$ one can easily obtain the inequality
\[
|m_{ij}^{(1)}(t, x, y)| < \infty\quad\text{for~all}~ t \ge 0,
\]
(see, e.g., the proof of similar facts in \cite{YarBRW:e,MN:12:e}), but in order not to overload the exposition, its elementary proof is given in Section~\ref{solution_first_moments}.

Nevertheless, let us explain the main ideas of the corresponding proof. Equation~\eqref{El21-1} with initial condition~\eqref{El21-initcind} can be treated as a linear differential equation in a Banach space whose right-hand side (for each $t$ and $y$) is a linear bounded operator acting in any of the spaces $l_{p}(\mathbb{Z}^{d})$, $p\ge1$. Since in this case the initial condition $m_{ij}^{(1)} (0,x,y)$ for each $y$ as a function of the variable $x$ also belongs to each of the spaces $l_{p}(\mathbb{Z}^{d})$, $p\ge1$, then, as shown, for example, in \cite{YarBRW:e,MN:12:e}, $m_{ij}^{(1)} (t,x,y)$ (for each $t$ and $y$) as a function of the variable $x$ also belongs to each of the spaces $l_{p}(\mathbb{Z}^{d})$, $p\ge1$, and is thus bounded.
\end{remark}

In Lemma~\ref{lemma_first_moments}, we have obtained the differential equations for the subpopulations generated by a single particle of each type. Now, we want to obtain the differential equation for the full population $N(t,y)$.

Define $m_i^{(1)} (t,x,y) = \Expect n_i (t,x,y)$, $i=1,2$, and rewrite the Equations~\eqref{El21-1} from Lemma~\ref{lemma_first_moments} in the following form:
\begin{align}\notag
\frac{\partial m_1^{(1)} (t,x,y)}{\partial t} &= (\mathcal L_1 m_1^{(1)} (t,\cdot,y))(x) + \sum_{k+l\geqslant2} l \beta_1 (k,l) m_2^{(1)} (t,x,y) \notag\\&\quad+ \Bigl(\sum_{k+l\geqslant2} (k-1) \beta_1 (k,l) - \mu_1\Bigr) m_1^{(1)} (t,x,y),\label{m_1(t,x,y)}
\end{align}
\begin{align}\notag
\frac{\partial m_2^{(1)} (t,x,y)}{\partial t} &= (\mathcal L_2 m_2^{(1)} (t,\cdot,y))(x) + \sum_{k+l\geqslant2} k \beta_2 (k,l) m_1^{(1)} (t,x,y)\notag \\&\quad+ \Bigl(\sum_{k+l\geqslant2} (l-1) \beta_2 (k,l)-\mu_2\Bigr) m_2^{(1)} (t,x,y),\label{m_2(t,x,y)}
\end{align}
with the initial conditions
\[
m_1^{(1)} (0,x,y) = [\delta_x (y),0]^{T},\qquad m_2^{(1)} (0,x,y) = [0,\delta_x (y)]^{T}.
\]

Let us denote $n(t,x,y) = [n_1 (t,x,y),n_2 (t,x,y)]^T$ and $ m^{(1)} (t,x,y) = \Expect n (t,x,y)$. Then, the pair of equations,  Equations~\eqref{m_1(t,x,y)} and \eqref{m_2(t,x,y)}, can be rewritten in a more compact~form:
\[
\frac{\partial m^{(1)} (t,x,y)}{\partial t} = \begin{pmatrix}
(\mathcal L_1 m_1^{(1)} (t,\cdot,y))(x)\\
(\mathcal L_2 m_2^{(1)} (t,\cdot,y))(x)\end{pmatrix} + V\begin{pmatrix}
m_1^{(1)}\\
m_2^{(1)}\end{pmatrix},
\]
where $V$ is the matrix
\[
V = \begin{pmatrix}
-\mu_1 + \sum_{k+l\geqslant2} (k-1) \beta_1 (k,l) & \sum_{k+l\geqslant2} l \beta_1 (k,l)\\
\sum_{k+l\geqslant2} k \beta_2 (k,l) & -\mu_2 + \sum_{k+l\geqslant2} (l-1) \beta_2 (k,l)
\end{pmatrix}.
\]

The above calculus let us the opportunity to get the equation for the full population at the site $y\in\mathbb Z^d$. Using the representation of $N_i (t,y)$, $i=1,2$, in Equation~\eqref{N_i(t,y)} we obtain for $m^{(1)} (t,y) := \Expect N(t,y) = [m_1^{(1)} (t,y),m_2^{(1)} (t,y)]^T$ the following formula:
\[
m_i^{(1)} (t,y) = \Expect N_i (t,y) = \sum_{x\in\mathbb Z^d}\sum_{s\in\{ 1,\ldots,l_1 \}} m_{1i,s}^{(1)} (t,x,y) + \sum_{x\in\mathbb Z^d}\sum_{m\in\{ 1,\ldots,l_2 \}} m_{2i,m}^{(1)} (t,x,y).
\]

Taking the partial derivative over the parameter $t$ for each component of $m^{(1)} (t,y)$ we derive from the above formula the equation:
\begin{equation}\label{E:dm1ty}
\frac{\partial m^{(1)} (t,y)}{\partial t} = \sum_{x\in\mathbb Z^d}\sum_{s\in\{ 1,\ldots,l_1 \}} \frac{\partial m_{1,s}^{(1)} (t,x,y)}{\partial t} + \sum_{x\in\mathbb Z^d}\sum_{m\in\{ 1,\ldots,l_2 \}} \frac{\partial m_{2,m}^{(1)} (t,x,y)}{\partial t},
\end{equation}
where $m_{i,l}^{(1)} (t,x,y) := \Expect n_{i,l} (t,x,y)$.

Formula~\eqref{E:dm1ty} describes how the behavior of the full population depends on the behavior of each subpopulation. Later on, we will study the behavior of subpopulations in more~details.

\subsection{Solutions of Differential Equations for the First Moments}
\label{solution_first_moments}

In this section, we will find an explicit form of the solutions of the differential equations obtained in Lemma~\ref{lemma_first_moments}. To find these solutions, we will use the discrete Fourier transform. For simplicity, we recall that the Fourier transform $\widehat{f}(\theta)$ of a function $f(u)$ is defined as
\begin{equation}\label{fourier_transform}
\widehat{f}(\theta) = \sum_{u\in\mathbb Z^d} e^{i (\theta,u)} f(u),\qquad \theta\in[-\pi,\pi]^d,
\end{equation}
where $(\cdot,\cdot)$ is the dot product in $\mathbb R^d\times\mathbb R^d$, while the inverse Fourier transform is of the form
\begin{equation}\label{inv_fourier_transform}
f(u) = \left(\frac{1}{2\pi}\right)^d \int_{[-\pi,\pi]^d} \widehat{f}(\theta) e^{-i (\theta,u)}\,d\theta.
\end{equation}

By applying the Fourier transform~\eqref{fourier_transform} to Equations~\eqref{m_1(t,x,y)} and~\eqref{m_2(t,x,y)}, we obtain the~equations
\begin{align}\notag
\frac{\partial \widehat{m}_1^{(1)} (t,\theta,y)}{\partial t} &= \varkappa_1 \widehat{a}_1 (\theta) \widehat{m}_1^{(1)} (t,\theta,y)  + \sum_{k+l\geqslant2} l \beta_1 (k,l) \widehat{m}_2^{(1)} (t,\theta,y) \notag\\&\quad+ \Bigl(\sum_{k+l\geqslant2} (k-1) \beta_1 (k,l) - \mu_1\Bigr) \widehat{m}_1^{(1)} (t,\theta,y),\label{fourier_m_1(t,x,y)}
\end{align}
\begin{align}\notag
\frac{\partial \widehat{m}_2^{(1)} (t,\theta,y)}{\partial t} &= \varkappa_2 \widehat{a}_2 (\theta) \widehat{m}_2^{(1)} (t,\theta,y) + \sum_{k+l\geqslant2} k \beta_2 (k,l) \widehat{m}_1^{(1)} (t,\theta,y) \notag\\&\quad+ \Bigl(\sum_{k+l\geqslant2} (l-1) \beta_2 (k,l) - \mu_2\Bigr) \widehat{m}_2^{(1)} (t,\theta,y),\label{fourier_m_2(t,x,y)}
\end{align}
with the initial conditions
\[
\widehat{m}_1^{(1)} (0,\theta,y) = [e^{i (\theta,y)},0]^{T},\qquad \widehat{m}_2^{(1)} (0,\theta,y) = [0,e^{i (\theta,y)}]^{T}.
\]

To simplify formulas~\eqref{fourier_m_1(t,x,y)} and~\eqref{fourier_m_2(t,x,y)}, let us introduce the following notations:
\begin{align}\label{a}
a(\theta) &= \varkappa_1 \widehat{a}_1 (\theta) + \Bigl(\sum_{k+l\geqslant2} (k-1) \beta_1 (k,l) - \mu_1\Bigr);\\
\label{b}
b &= \sum_{k+l\geqslant2} l \beta_1 (k,l) \geqslant0;\\
\label{c}
c &= \sum_{k+l\geqslant2} k \beta_2 (k,l) \geqslant0;\\
\label{d}
d(\theta) &= \varkappa_2 \widehat{a}_2 (\theta) + \Bigl(\sum_{k+l\geqslant2} (l-1) \beta_2 (k,l) - \mu_2\Bigr).
\end{align}

With the usage of these notations, Equations~\eqref{fourier_m_1(t,x,y)} and~\eqref{fourier_m_2(t,x,y)} can be represented in a more compact form:
\begin{alignat}{2}\label{fourier_m_1(t,x,y)_new}
\frac{\partial \widehat{m}_1^{(1)} (t,\theta,y)}{\partial t} &= a(\theta) \widehat{m}_1^{(1)} (t,\theta,y) + b \widehat{m}_2^{(1)} (t,\theta,y),&\qquad
\widehat{m}_1^{(1)} (0,\theta,y) &= [e^{i (\theta,y)},~0]^{T}\\
\label{fourier_m_2(t,x,y)_new}
\frac{\partial \widehat{m}_2^{(1)} (t,\theta,y)}{\partial t} &= c \widehat{m}_1^{(1)} (t,\theta,y) + d(\theta) \widehat{m}_2^{(1)} (t,\theta,y),&\qquad
\widehat{m}_2^{(1)} (0,\theta,y) &= [0,e^{i (\theta,y)}]^{T}.
\end{alignat}

To get a solution for this last system of differential equations, let us recall some facts from the theory of two-dimensional linear differential equations and perform some auxiliary calculations.

\begin{remark}\label{Yrem_ODE}
Represent Equations~\eqref{fourier_m_1(t,x,y)_new} and \eqref{fourier_m_2(t,x,y)_new} arising in our treatment in a conventional form of a system of linear differential equations with two variables (see details in~\cite{Filipp}):
\begin{alignat}{2}\label{Yu(t)}
\frac{du(t)}{dt} &= au(t) + bv(t),&\qquad u(0)&=u_{0},\\
\label{Yv(t)}
\frac{dv(t)}{dt} &= cu(t) + dv(t),&\qquad v(0)&=v_{0},
\end{alignat}
assuming that $a,b,c$ and $d$ here are some numerical parameters. In order to ``keep the connection'' with Equations~\eqref{fourier_m_1(t,x,y)_new} and \eqref{fourier_m_2(t,x,y)_new} and not consider options unnecessary in the future, we will assume throughout this remark that
\[
b,c\geqslant0.
\]

As is known (see, e.g.,~\cite{Filipp} or some other handbook on the theory of differential equations), the behavior of solutions of Equations~\eqref{Yu(t)} and \eqref{Yv(t)} is completely determined, in a sense, by the roots of characteristic equation of the matrix of coefficients standing in the right-hand side of Equations~\eqref{Yu(t)} and \eqref{Yv(t)}:
\begin{equation}\label{char_eq_rem}
\lambda^2 - (a+d)\lambda + (ad-bc) = 0.
\end{equation}
These roots are as follows
\[
  \lambda_{1}=\frac{a+d+\sqrt{D}}{2},\quad\lambda_{2}=\frac{a+d-\sqrt{D}}{2}\quad\text{where}~D = (a-d)^2 + 4bc.
\]
Note that under the assumption $b,c\geqslant 0$, the discriminant $D$ is non-negative, and therefore the roots $\lambda_{1,2}$ are real.

Let $D=0$; this can be if and only if
\[
a=d\quad\text{and}\quad b=0~\text{or}~c=0.
\]
In this case, $\lambda_{1}=\lambda_{2}$ coincide with each other and moreover $\lambda_{1} = \lambda_{2}=a=d$. Then~(see, e.g.,~\cite{Filipp}), the solution $u(t)$ of Equations~\eqref{Yu(t)} and \eqref{Yv(t)} is a linear combination of the functions $e^{\lambda t}$ and $te^{\lambda t}$:
\[
u(t) = (C_1 + C_2 t)e^{\lambda t}.
\]
The solution $v(t)$ can be expressed likewise.

Let $D\ne0$; this can be if and only if
\[
a\neq d\quad\text{or}\quad b\neq0~\text{and}~c\neq0.
\]
In this case, $\lambda_{1}\neq\lambda_{2}$ and~(see, e.g.,~\cite{Filipp}) the solution $u(t)$ of Equations~\eqref{Yu(t)}  and \eqref{Yv(t)} is a linear combination of the functions $e^{\lambda_{1}t}$ and $e^{\lambda_{2}t}$:
\begin{equation}\label{Yrem_eq_roots}
u(t) = C_1 e^{\lambda_{1}t} + C_2 e^{\lambda_{2}t}.
\end{equation}
The solution $v(t)$ can be expressed likewise.

Let us write out the precise forms of the solutions $u(t)$ and $v(t)$ of Equations~\eqref{Yu(t)} and \eqref{Yv(t)}; they will be needed in the further analysis. Consider the following combinations of the parameters $b$ and $c$:
\[
b=0,c\geqslant0\quad\text{or}\quad b\geqslant0,c=0\quad\text{or}\quad b>0,c>0,
\]
which exhaust all possible combinations of these parameters under condition $b,c\geqslant0$. The fact that the first two of these conditions intersect does not interfere with further considerations. We also mention that the case $b=c=0$ is covered by both of the first two cases.

\paragraph{Case $b=0,c\geqslant0$.} Here, $u(t)$ can be found directly from Equation~\eqref{Yu(t)}:
\[
u(t)=e^{at} u_{0}.
\]
To find $v(t)$ it suffices to substitute the obtained expression for $u(t)$ into Equation~\eqref{Yv(t)} and to solve the resulted non-homogeneous linear differential equation:
\[
v(t)=e^{dt} v_{0}+\int_{0}^{t}e^{d(t-s)}ce^{as} u_{0}\,ds.
\]
The value of the integral in the right-hand side of the obtained equality is different  depending on whether the equality $a=d$ holds. Direct evaluation shows that
\[
v(t)=\begin{cases}
  \left(v_{0}-\frac{cu_{0}}{a-d}\right)e^{dt}+\frac{cu_{0}}{a-d}e^{at},&\quad a\neq d\\
  (v_{0}+cu_{0} t)e^{dt},&\quad a= d.
\end{cases}
\]

\paragraph{Case $b\geqslant0,c=0$.} This case is treated similarly to the previous one, and we get:
\begin{align*}
u(t)&=\begin{cases}
  \left(u_{0}-\frac{bv_{0}}{d-a}\right)e^{at}+\frac{bv_{0}}{d-a}e^{dt},&\quad a\neq d\\
  (u_{0}+bv_{0} t)e^{at},&\quad a= d,
\end{cases}\\
v(t)&=e^{dt} v_{0}.
\end{align*}

\paragraph{Case $b>0,c>0$.} In this case, both roots $\lambda_{1,2}$ of Equation~\eqref{char_eq_rem} are different, and moreover $\lambda_1>\lambda_2$. In order to find the solutions $u(t)$ and $v(t)$ of Equations~\eqref{Yu(t)} and \eqref{Yv(t)} let us first take $t=0$ in  Equation~\eqref{Yrem_eq_roots}. Then, we obtain the following equation for the initial condition $u(0)$:
\begin{equation}\label{YC1C2}
C_1 + C_2 = u(0)=u_{0}.
\end{equation}
Further, find $bv(t)$ from Equation~\eqref{Yu(t)}:
\[
bv(t)=u'(t)-au(t)= C_{1}(\lambda_{1}-a)e^{\lambda_{1}t}+C_{2}(\lambda_{2}-a)e^{\lambda_{2}t}.
\]
Using the obtained expression we will get the equation for $bv(0)$:
\begin{equation}\label{YCL1CL2}
C_{1}(\lambda_{1}-a)+C_{2}(\lambda_{2}-a)=bv(0)=bv_{0}.
\end{equation}
By solving the resulting system of Equations~\eqref{YC1C2} and \eqref{YCL1CL2}, we get
\[
C_1 = \frac{bv_0 + u_0 (a - \lambda_2)}{\lambda_1 - \lambda_2},\qquad
C_2 = \frac{u_0 (\lambda_1 - a) - bv_0}{\lambda_1 - \lambda_2},
\]
from which
\begin{align*}
u(t) &= \frac{bv_0 + u_0 (a - \lambda_2)}{\lambda_1 - \lambda_2} e^{\lambda_1 t} + \frac{u_0 (\lambda_1 - a) - bv_0}{\lambda_1 - \lambda_2} e^{\lambda_2 t},\\
v(t) &= \frac{(\lambda_1 - a)}{b}\frac{bv_0 + u_0 (a - \lambda_2)}{\lambda_1 - \lambda_2} e^{\lambda_1 t} + \frac{(\lambda_2 - a)}{b}\frac{u_0 (\lambda_1 - a) - bv_0}{\lambda_1 - \lambda_2} e^{\lambda_2 t}.
\end{align*}

Here, the last equation can be simplified by noting that
\[
\frac{(\lambda_{2}-a)(\lambda_{1}-a)}{b}=-c,\quad \lambda_1-\lambda_2=\sqrt{D}.
\]

As a result, we obtain:
\begin{align*}
u(t) &= \frac{1}{\sqrt{D}} \biggl(\bigl(bv_0 + u_0 (a - \lambda_2)\bigr) e^{\lambda_1 t} + \bigl(u_0 (\lambda_1 - a) - bv_0\bigr) e^{\lambda_2 t} \biggr),\\
v(t) &= \frac{1}{\sqrt{D}} \biggl(\bigl(v_0(\lambda_1 - a) + cu_0\bigr)e^{\lambda_1 t} - \bigl(cu_0 - (\lambda_2 - a)v_0 \bigr)e^{\lambda_2 t}\biggr).
\end{align*}
\qed
\end{remark}

Now, we are able to write out the solutions of Equations~\eqref{fourier_m_1(t,x,y)_new} and \eqref{fourier_m_2(t,x,y)_new}. For this, it suffices to note that although in the reasoning of Remark~\ref{Yrem_ODE}  it was implicitly assumed that the functions $u(t)$ and $v(t)$ are scalar,  but in fact this assumption was never used anywhere, and the functions $u(t)$ and $v(t)$ may be assumed vector-valued, for example, such as $\widehat{m}_i^{(1)} (t,\theta,y)$ in Equations~\eqref{fourier_m_1(t,x,y)_new} and \eqref{fourier_m_2(t,x,y)_new}.

One should also pay attention to the fact that in Equations~\eqref{fourier_m_1(t,x,y)_new} and \eqref{fourier_m_2(t,x,y)_new}, in contrast to Equations~\eqref{Yu(t)} and \eqref{Yv(t)}, the parameters $a$ and $d$ are actually functions of the variable $\theta$, that is, $a=a(\theta)$ and $d=d(\theta)$, and then the values $\lambda_{1}, \lambda_{2}$ and $D$ are also functions of the variable $\theta$:
\begin{equation}\label{lambda}
  \lambda_{1}(\theta)=\frac{a(\theta)+d(\theta)+\sqrt{D(\theta)}}{2},\quad
  \lambda_{2}(\theta)=\frac{a(\theta)+d(\theta)-\sqrt{D(\theta)}}{2},
\end{equation}
and
\begin{equation}\label{D}
D(\theta) = \left(a(\theta)-d(\theta)\right)^2 + 4bc.
\end{equation}

Considering the above, we can write out the solutions $\widehat{m}_1^{(1)}(t,\theta,y)$ and $\widehat{m}_2^{(1)}(t,\theta,y)$ of Equations~\eqref{fourier_m_1(t,x,y)_new} and \eqref{fourier_m_2(t,x,y)_new} using the appropriate initial conditions.

\paragraph{Case $b=0,c\geqslant0$.} Here,
\begin{align*}
\widehat{m}_1^{(1)}(t,\theta,y)&=e^{a(\theta)t} \widehat{m}_1^{(1)}(0,\theta,y),\\
\widehat{m}_2^{(1)} (t,\theta,y)&=\begin{cases}
  \left(\widehat{m}_2^{(1)}(0,\theta,y)-\frac{c}{a(\theta)-d(\theta)}\widehat{m}_1^{(1)}(0,\theta,y)\right)e^{d(\theta)t}+\\ \qquad+\frac{c}{a(\theta)-d(\theta)}\widehat{m}_1^{(1)}(0,\theta,y)e^{a(\theta)t},& \quad\text{if}~\theta~\text{s.t.}~ a(\theta)\neq d(\theta),\\
  (\widehat{m}_2^{(1)}(0,\theta,y)+c\widehat{m}_1^{(1)}(0,\theta,y) t)e^{dt},&\quad\text{if}~\theta~\text{s.t.}~ a(\theta)= d(\theta).
\end{cases}
\end{align*}

\paragraph{Case $b\geqslant0,c=0$.} Here,
\begin{align*}
\widehat{m}_1^{(1)}(t,\theta,y)&=\begin{cases}
  \left(\widehat{m}_1^{(1)}(0,\theta,y)-\frac{b}{d(\theta)-a(\theta)}\widehat{m}_2^{(1)}(0,\theta,y)\right)e^{a(\theta)t}+\\
  \qquad+\frac{b}{d(\theta)-a(\theta)}\widehat{m}_2^{(1)}(0,\theta,y)e^{d(\theta)t},&\quad\text{if}~\theta~\text{s.t.}~ a(\theta)\neq d(\theta),\\
  (\widehat{m}_1^{(1)}(0,\theta,y)+b\widehat{m}_2^{(1)}(0,\theta,y) t)e^{a(\theta)t},&\quad\text{if}~\theta~\text{s.t.}~ a(\theta)= d(\theta),
\end{cases}\\
\widehat{m}_2^{(1)} (t,\theta,y)&=e^{d(\theta)t} \widehat{m}_2^{(1)} (0,\theta,y).
\end{align*}

\paragraph{Case $b>0,c>0$.} Here,
\begin{align*}
\widehat{m}_1^{(1)}(t,\theta,y) &= \frac{1}{\sqrt{D(\theta)}}\left((a(\theta)-\lambda_{2}(\theta))\widehat{m}_1^{(1)}(0,\theta,y)+b\widehat{m}_2^{(1)}(0,\theta,y)\right) e^{\lambda_{1}(\theta)t} +\\
&\qquad+ \frac{1}{\sqrt{D(\theta)}}\left((\lambda_{1}(\theta)-a(\theta))\widehat{m}_1^{(1)}(0,\theta,y)-b\widehat{m}_2^{(1)}(0,\theta,y)\right) e^{\lambda_{2}(\theta)t}\\
\widehat{m}_2^{(1)}(t,\theta,y)&= \frac{1}{\sqrt{D(\theta)}}\left(c\widehat{m}_1^{(1)}(0,\theta,y)+(\lambda_{1}(\theta)-a(\theta))\widehat{m}_2^{(1)}(0,\theta,y)\right) e^{\lambda_{1}(\theta)t}+\\
&\qquad+
\frac{1}{\sqrt{D(\theta)}}\left(-c\widehat{m}_1^{(1)}(0,\theta,y)+(\lambda_{2}(\theta)-a(\theta))\widehat{m}_2^{(1)}(0,\theta,y)\right)e^{\lambda_{2}(\theta)t}.
\end{align*}

Finally, it needs to be remembered that each of the functions $\widehat{m}_1^{(1)}(t,\theta,y)$ and $\widehat{m}_2^{(1)}(t,\theta,y)$ is a two-component vector-function. Therefore, taking the components of the functions $\widehat{m}_1^{(1)}(t,\theta,y)$ and $\widehat{m}_2^{(1)}(t,\theta,y)$ for the obtained above three cases, we will obtain three cases of formulas for $\widehat{m}_{11}^{(1)}(t,\theta,y)$, $\widehat{m}_{12}^{(1)}(t,\theta,y)$, $\widehat{m}_{21}^{(1)}(t,\theta,y)$ and $\widehat{m}_{22}^{(1)}(t,\theta,y)$:

\paragraph{Case $b=0,c\geqslant0$.} Here,
\begin{equation}\label{m_ij_b=0_c>0}
\left.\begin{aligned}
\widehat{m}_{11}^{(1)} (t,\theta,y) &= e^{i (\theta,y)}e^{a(\theta)t};\\
\widehat{m}_{21}^{(1)} (t,\theta,y) &=
\begin{cases}
	\frac{c}{a(\theta)-d(\theta)}e^{i (\theta,y)}\bigl(e^{a(\theta)t} - e^{d(\theta)t}\bigr),&\text{if}~\theta~\text{s.t.}~a(\theta)\ne d(\theta),\\
	cte^{i (\theta,y)} e^{a(\theta)t},&\text{if}~\theta~\text{s.t.}~a(\theta)=d(\theta);\
\end{cases}\\
\widehat{m}_{12}^{(1)} (t,\theta,y) &= 0;\\
\widehat{m}_{22}^{(1)} (t,\theta,y) &= e^{i (\theta,y)} e^{d(\theta)t}.
\end{aligned}\right\}
\end{equation}

\paragraph{Case $b\geqslant0,c=0$.} Here,
\begin{equation}\label{m_ij_b>0_c=0}
\left.\begin{aligned}
\widehat{m}_{11}^{(1)} (t,\theta,y) &= e^{i (\theta,y)}e^{a(\theta)t};\\
\widehat{m}_{21}^{(1)} (t,\theta,y) &= 0;\\
\widehat{m}_{12}^{(1)} (t,\theta,y) &=
\begin{cases}
	\frac{b}{a(\theta)-d(\theta)}e^{i (\theta,y)}\bigl(e^{a(\theta)t} - e^{d(\theta)t}\bigr),&\text{if}~\theta~\text{s.t.}~a(\theta)\ne d(\theta),\\
	bte^{i (\theta,y)} e^{d(\theta)t},&\text{if}~\theta~\text{s.t.}~a(\theta)=d(\theta);
\end{cases}\\
\widehat{m}_{22}^{(1)} (t,\theta,y) &= e^{i (\theta,y)} e^{d(\theta)t}.
\end{aligned}\right\}
\end{equation}

\paragraph{Case $b>0,c>0$.} Here,
\begin{equation}\label{m_ij_b>0_c>0}
\left.\begin{aligned}
	\widehat{m}_{11}^{(1)} (t,\theta,y) &= \frac{1}{\sqrt{D(\theta)}}e^{i (\theta,y)} \bigl(\bigl(a(\theta) - \lambda_2(\theta)\bigr) e^{\lambda_1(\theta) t} + \bigl(\lambda_1(\theta) - a(\theta)\bigr) e^{\lambda_2(\theta) t} \bigr);\\
	\widehat{m}_{21}^{(1)} (t,\theta,y) &= \frac{c}{\sqrt{D(\theta)}}e^{i (\theta,y)} \bigl(e^{\lambda_1(\theta) t} - e^{\lambda_2(\theta) t}\bigr);\\
	\widehat{m}_{12}^{(1)} (t,\theta,y) &= \frac{b}{\sqrt{D(\theta)}}e^{i (\theta,y)} \bigl(e^{\lambda_1(\theta) t} - e^{\lambda_2(\theta) t} \bigr);\\
	\widehat{m}_{22}^{(1)} (t,\theta,y) &= \frac{1}{\sqrt{D(\theta)}}e^{i (\theta,y)} \bigl(\bigl(\lambda_1(\theta) - a(\theta)\bigr)e^{\lambda_1(\theta) t} + \bigl(\lambda_2(\theta) - a(\theta)\bigr)e^{\lambda_2(\theta) t}\bigr).
\end{aligned}\right\}
\end{equation}

\begin{remark}
The cases $b=0,c>0$ or $b>0,c=0$ can describe the situation when particles of one type cannot produce offsprings of both types. This can have the real-life interpretation: we have some species which have both ``male'' and ``female'' individuals and the ``male'' individuals cannot produce offspring.
\end{remark}

\begin{remark}
The attentive reader will notice that our constructions are redundant in a sense. In the middle of this section, we made an effort to go from equations for the functions $m_{ij}^{(1)}(t,x,y)$, $i,j=1,2$, to more general equations for the functions $m_{1}^{(1)} (t,x,y)$ and $m_{2}^{(1)} (t,x,y)$, and their Fourier transforms $\widehat{m}_{1}^{(1)} (t,\theta,y)$ and $\widehat{m}_{2}^{(1)} (t,\theta,y)$. Then, at the end of this section, we again return to the functions $m_{ij}^{(1)} (t,x,y)$, $i,j=1,2$ (or rather to their Fourier images $\widehat{m}_{ij}^{(1)} (t,\theta,y)$, $i,j=1,2$). We emphasize once again that from a technical point of view, this method of research is redundant, however, in our opinion, it contributes to a deeper understanding of the ``nature of things'' when analyzing the behavior of the functions $m_{ij}^{(1)} (t,x,y)$, $i,j=1,2$.
\end{remark}

So, we have found the solutions of Equations~\eqref{fourier_m_1(t,x,y)_new} and \eqref{fourier_m_2(t,x,y)_new}. Applying the inverse Fourier transform~\eqref{inv_fourier_transform} to Equations~\eqref{m_ij_b>0_c>0}, we can get the solutions for $m_{ij}^{(1)} (t,x,y)$, $i,j=1,2$. Later, we will find the asymptotic behavior of each subpopulation $m_{ij}^{(1)} (t,x,y)$ in one particular case.

\subsection{Asymptotic Behavior in the Case of Finite Variance of Jumps}\label{E:AsympFVJ}

In the previous section, in Equations~\eqref{m_ij_b=0_c>0}--\eqref{m_ij_b>0_c>0}, we found the solutions for the Fourier transform of the first moments of the subpopulations $\widehat{m}_{ij}^{(1)} (t,\theta,y)$, $i,j=1,2$. In this section, we obtain their asymptotic behavior in one particular case that is natural for applications.

\begin{remark}\label{remark_p(t,x,y)}
Consider the parabolic problem
\begin{equation}\label{parabolic_problem}
\frac{\partial p(t,x,y)}{\partial t} = (\mathcal L_i p(t,\cdot,y))(x),\qquad
p(0,x,y) = \delta_x (y)
\end{equation}
where operators $\mathcal L_i$, $i=1,2$, are defined in~\eqref{generator}.

By applying the discrete Fourier transform~\eqref{fourier_transform} to Equation~\eqref{parabolic_problem}, we find that the Fourier image $\widehat{p}(t,\theta,y)$ of the function $p(t,x,y)$ satisfies the Cauchy problem
\begin{equation}\label{solution_parabolic_problem}
\frac{\partial \widehat{p}(t,\theta,y)}{\partial t} = \varkappa \widehat{a}_i(\theta) \widehat{p}(t,\theta,y),\qquad
\widehat{p}(0,\theta,y) = e^{i(\theta,y)},
\end{equation}
whose solution can be found explicitly:
\begin{equation}\label{fourier_solution_parabolic_problem}
\widehat{p}(t,\theta,y) = e^{i (\theta,y)} e^{\varkappa\widehat{a}_i(\theta)t}.
\end{equation}
Applying the inverse Fourier transform to Equation~\eqref{fourier_solution_parabolic_problem}, we obtain the solution of Equation~\eqref{parabolic_problem}:
\begin{equation}\label{E-pi}
p(t,x,y)=\frac{1}{(2\pi)^{d}}\int_{[-\pi,\pi]^{d}}e^{\varkappa\widehat{a}_i(\theta)t+i(\theta,y-x)}\,d\theta
\end{equation}
Besides, from here we can see that $p(t,x,y)$ depends only on $x-y$, which gives an alternative proof to the corresponding assertion from Remark~\ref{R-uniform}.
\end{remark}

Now, turn to the problem of finding $m_{ij}^{(1)} (t,x,y)$, $i,j=1,2$. Let $a_1 (v) = a_2 (v) =: a_*(v)$ for all $v\in\mathbb Z^d$ and $\varkappa_1 = \varkappa_2 = \varkappa$, so the migration operators from Equation~\eqref{generator} are equal. Besides, consider the case when underlying random walk has finite variance of jumps, so
\begin{equation}\label{f_v_j}
\sum_{v\ne0} a_*(v)|v|^2 < \infty,
\end{equation}
where $|\cdot|$ is the vector norm in $\mathbb R^d$.

As was demonstrated, e.g., in~\cite{YarBRW:e} under condition~\eqref{f_v_j}, the solution of the parabolic problem~\eqref{parabolic_problem} has, for each $x,y\in\mathbb{Z}^{d}$, the following asymptotics:
\begin{equation}\label{asymptotic_f_v_j}
p(t,x,y) \sim \frac{\gamma_d}{t^{d/2}},\qquad  t\to\infty,
\end{equation}
where
\begin{equation}\label{gamma_d}
\gamma_d = \Bigl((2\pi)^d |\det (\varkappa \widehat{a}''_* (0))|\Bigr)^{-1/2}
\end{equation}
is a constant depending on the lattice dimension. For a more detailed description of asymptotics~\eqref{asymptotic_f_v_j}, including the form of the reminder term, see~\cite{MY12}.

Let us now apply the above reasoning to Equations~\eqref{fourier_m_1(t,x,y)_new} and \eqref{fourier_m_2(t,x,y)_new}.
Note that in the case where the migration operators $\mathcal L_1$ and $\mathcal L_2$ defined by Equation~\eqref{generator} coincide, i.e., $\mathcal L_1=\mathcal L_2$, we can refine the representation of Equation~\eqref{lambda} for $\lambda_{1}(\theta)$ and $\lambda_{2}(\theta)$ by using Equation~\eqref{D}, which yields:
\[
\lambda_{1,2} (\theta) = \frac{a(\theta)+d(\theta)}{2} \pm \frac{((a(\theta)-d(\theta))^2 + 4bc)^{1/2}}{2} = \varkappa \widehat{a} (\theta) + C_1 \pm C_2,
\]
where
\begin{equation}\label{C_1_C_2}
C_1 = \frac{a(\theta)+d(\theta)}{2} - \varkappa \widehat{a}(\theta),\quad C_2 = \frac{((a(\theta)-d(\theta))^2 + 4bc)^{1/2}}{2}.
\end{equation}
Replacing $a(\theta),b,c$ and $d(\theta)$ in Equation~\eqref{C_1_C_2} by their values given by Equations~\eqref{a}--\eqref{d} we obtain the following representations of $C_1$ and $C_2$:
\begin{align*}
C_1 &= \frac{1}{2}\sum_{k+l\geqslant2} \bigl[(k-1)\beta_1 (k,l) + (l-1)\beta_2 (k,l)\bigr] - (\mu_1 + \mu_2);\\
C_2 &= \frac{1}{2} \biggl[ \Bigl( \sum_{k+l\geqslant2} \bigl[(k-1)\beta_1 (k,l) - (l-1)\beta_2 (k,l)\bigr] - (\mu_1 - \mu_2) \Bigr)^2 \\&\quad+ 4\Bigl( \sum_{k+l\geqslant2} l\beta_1 (k,l) \Bigr)\Bigl( \sum_{k+l\geqslant2} k\beta_2 (k,l) \Bigr) \biggr]^{1/2}.
\end{align*}

Let us denote
\[
r_1 = \sum_{k+l\geqslant2} (k-1)\beta_1 (k,l) - \mu_1,\quad
r_2 = \sum_{k+l\geqslant2} (l-1)\beta_2 (k,l) - \mu_2.
\]
Then, in the case when $\varkappa_1 = \varkappa_2=\varkappa$, $b=0$, $c>0$ (or $b>0$, $c=0$) due to Equations~\eqref{a} and \eqref{d}, the following relation holds:
\[
a(\theta)-d(\theta)=r_{1}-r_{2}\quad\text{for all}~\theta,
\]
that is, the difference $a(\theta)-d(\theta)$ does not depend on $\theta$. This means that either $a(\theta)-d(\theta)=0$ for all values of $\theta$ or $a(\theta)-d(\theta)\neq0$ also for all values of $\theta$. Moreover,
\[
a(\theta) - d(\theta) = 0~\text{for all}~\theta\quad \Longleftrightarrow\quad r_1 - r_2 = 0\quad \Longleftrightarrow\quad C_2 = 0.
\]
Consequently, in the case $r_{1}=r_{2}$ we have not only that $a_1 (v) = a_2 (v)$ for all $v\in\mathbb Z^d$, but also $\widehat{a}_1 (\theta) = \widehat{a}_2 (\theta)$ for all $\theta\in[-\pi,\pi]^d$.

Thus, from Equations~\eqref{m_ij_b=0_c>0}--\eqref{m_ij_b>0_c>0} we have for $t\to\infty$ (we prefer to consider the case $b=c=0$ separately from other cases):
\paragraph{Case $b=0,c=0$.} Here, for each $x,y\in\mathbb{Z}^{d}$,
\begin{equation*}
\begin{aligned}
	m_{11}^{(1)} (t,x,y) &\sim e^{r_1 t}\frac{\gamma_d}{t^{d/2}};\\
	m_{21}^{(1)} (t,x,y) &= 0;\\
	m_{12}^{(1)} (t,x,y) &= 0;\\
	m_{22}^{(1)} (t,x,y) &\sim e^{r_2 t}\frac{\gamma_d}{t^{d/2}}.
\end{aligned}
\end{equation*}

\paragraph{Case $b=0,c>0$.} Here, for each $x,y\in\mathbb{Z}^{d}$,
\begin{equation*}
\begin{aligned}
	m_{11}^{(1)} (t,x,y) &\sim e^{r_1 t}\frac{\gamma_d}{t^{d/2}};\\
	m_{21}^{(1)} (t,x,y) &\sim \begin{cases}
		ce^{r_1 t}\frac{\gamma_d}{t^{d/2 - 1}},&\text{if}~C_2 = 0,\\
		\frac{c}{r_1 - r_2}\bigl(e^{r_1 t} - e^{r_2 t}\bigr)\frac{\gamma_d}{t^{d/2}},&\text{if}~C_2 \ne0;
		\end{cases}\\
	m_{12}^{(1)} (t,x,y) &= 0;\\
	m_{22}^{(1)} (t,x,y) &\sim e^{r_2 t}\frac{\gamma_d}{t^{d/2}}.
\end{aligned}
\end{equation*}

\paragraph{Case $b>0,c=0$.} Here, for each $x,y\in\mathbb{Z}^{d}$,
\begin{equation*}
\begin{aligned}
	m_{11}^{(1)} (t,x,y) &\sim e^{r_1 t}\frac{\gamma_d}{t^{d/2}};\\
	m_{21}^{(1)} (t,x,y) &= 0;\\
	m_{12}^{(1)} (t,x,y) &\sim \begin{cases}
		be^{r_2 t} \frac{\gamma_d}{t^{d/2 - 1}},&\text{if}~C_2 = 0,\\
		\frac{b}{r_1 - r_2}\bigl(e^{r_1 t} - e^{r_2 t}\bigr)\frac{\gamma_d}{t^{d/2}},&\text{if}~C_2 \ne0;
		\end{cases}\\
	m_{22}^{(1)} (t,x,y) &\sim e^{r_2 t}\frac{\gamma_d}{t^{d/2}}.
\end{aligned}
\end{equation*}

\paragraph{Case $b>0,c>0$.} Here, for each $x,y\in\mathbb{Z}^{d}$,
\begin{equation*}
\begin{aligned}
	m_{11}^{(1)} (t,x,y) &\sim \frac{e^{C_1 t}}{2C_2}\Bigl(\bigl(r_1 - C_1 + C_2\bigr) e^{C_2 t} + \bigl(C_1 + C_2 - r_1\bigr) e^{- C_2 t} \Bigr)\frac{\gamma_d}{t^{d/2}};\\
	m_{21}^{(1)} (t,x,y) &\sim \frac{ce^{C_1 t}}{2C_2} \Bigl(e^{C_2 t} - e^{- C_2 t}\Bigr)\frac{\gamma_d}{t^{d/2}};\\
	m_{12}^{(1)} (t,x,y) &\sim \frac{be^{C_1 t}}{2C_2}\Bigl(e^{C_2 t} - e^{ - C_2 t} \Bigr)\frac{\gamma_d}{t^{d/2}};\\
	m_{22}^{(1)} (t,x,y) &\sim \frac{e^{C_1 t}}{2C_2}\Bigl(C_1 + C_2 - r_1\bigr)e^{C_2 t} + \bigl(C_1 - C_2 - r_1\bigr)e^{- C_2 t}\Bigr)\frac{\gamma_d}{t^{d/2}}.
\end{aligned}
\end{equation*}

\section{The Second Moments}
\label{the_second_moments}

In this section, we will study the behavior of the second moments of the number of subpopulations. To do this, we will essentially use the technique developed in the previous section, so we will omit some technical details.

\subsection{Differential Equations for Moments}\label{E:DUM2}
Let us denote $m_{ij}^{(2)} (t,x,y) = \Expect n_{ij}^2 (t,x,y)$ and let the estimate Equation~\eqref{int_ineq} be true. Our goal in this section is to obtain differential equations for $m_{ij}^{(2)} (t,x,y)$, $i,j=1,2$, which are similar to those obtained for $m_{ij}^{(1)} (t,x,y)$, $i,j=1,2$, in Section~\ref{E:DUM1}.

By taking the partial derivatives of the functions $\Phi_i (t,x,y;z)$ in Equation~\eqref{subpop_gen_function} over $z_1$ and $z_2$ we can get the following equations:
\begin{align*}
\frac{\partial^2 \Phi_i (t,x,y;z)}{\partial z_j^2} &= \frac{\partial^2 \Expect z_1^{n_{i1} (t,x,y;z)} z_2^{n_{i2} (t,x,y)}}{\partial z_j^2} \notag\\&= \Expect n_{ij} (t,x,y)(n_{ij} (t,x,y) - 1)z_1^{n_{i1} (t,x,y) - 2\delta_j (1)} z_2^{n_{i2} (t,x,y) - 2\delta_j (2)}
\end{align*}

Then, by fixing $z_1 = z_2 = 1$ in the last equation, we obtain
\begin{equation}\label{second_factorial_moment}
\frac{\partial^2 \Phi_i (t,x,y;z)}{\partial z_j^2}\Bigr|_{z=(1,1)} = \Expect n_{ij} (t,x,y) (n_{ij} (t,x,y) - 1) = m_{ij}^{(2)} (t,x,y) - m_{ij}^{(1)} (t,x,y).
\end{equation}

Now, by differentiating both sides of Equation~\eqref{E:L21-1} from Lemma~\ref{lemma_gen_function} over $z_j$ twice, we obtain:
\begin{align*}
\frac{\partial^3 \Phi_i (t,x,y;z)}{\partial z_j^2 \partial t} &= \partial_{z_j z_j} \Bigl((\mathcal L_i \Phi_i (t,\cdot,y;z))(x)+ \mu_i (1-\Phi_i (t,x,y;z))\\&\quad + F_i (\Phi_1 (t,x,y;z),\Phi_2 (t,x,y;z))\Bigr).
\end{align*}
Taking here $z=(1,1)$ and using the notation
\[
m_{ij}^{(2!)} (t,x,y) = m_{ij}^{(2)} (t,x,y) - m_{ij}^{(1)} (t,x,y)
\]
we obtain the following representation of the left-hand side of Equation~\eqref{second_factorial_moment}:
\begin{equation}\label{second_factorial_moment_left_hand_part}
\frac{\partial^3 \Phi_i (t,x,y;z)}{\partial z_j^2 \partial t}\Bigr|_{z=(1,1)}= \frac{\partial}{\partial t} \frac{\partial^2 \Phi_i (t,x,y;z)}{\partial z_j^2}\Bigr|_{z=(1,1)} =
\frac{\partial m_{ij}^{(2!)} (t,x,y)}{\partial t}
\end{equation}
while the right-hand side of the same equation equals to:
\begin{align}\notag
&\partial_{z_j z_j} \Bigl( (\mathcal L_i \Phi_i (t,\cdot,y;z))(x) + \mu_i (1-\Phi_i (t,x,y;z))\\\notag
&\quad+ \sum_{k+l\geqslant2} \beta_i (k,l) (\Phi_1^k (t,x,y;z)\Phi_2^l (t,x,y;z) -\Phi_i (t,x,y;z)) \Bigr)\Bigr|_{z=(1,1)}\\
& = \Bigl( \bigl(\mathcal L_i (\partial_{z_j z_j} \Phi_i (t,\cdot,y;z))\bigr)(x) - \mu_i \bigl(\partial_{z_j z_j} \Phi_i (t,x,y;z)\bigr) + \sum_{k+l\geqslant2} \beta_i (k,l)\notag\\&\quad\times \Bigl( k(k-1)\bigl(\partial_{z_j} \Phi_1 (t,x,y;z)\bigr)^2 \Phi_1^{k-2} (t,x,y;z) \Phi_2^{l} (t,x,y;z) + k\bigl(\partial_{z_j z_j} \Phi_1 (t,x,y;z)\bigr)\notag\\&\quad\times\Phi_1^{k-1} (t,x,y;z) \Phi_2^{l} (t,x,y;z) + 2kl\bigl(\partial_{z_j} \Phi_1 (t,x,y;z)\bigr)\bigl(\partial_{z_j} \Phi_2 (t,x,y;z)\bigr)\Phi_1^{k-1} (t,x,y;z)\notag\\&\quad\times \Phi_2^{l-1} (t,x,y;z) + l(l-1)\bigl(\partial_{z_j} \Phi_2 (t,x,y;z)\bigr)^2 \Phi_1^{k} (t,x,y;z) \Phi_2^{l-2} (t,x,y;z)\notag \\&\quad+ l\bigl(\partial_{z_j z_j} \Phi_2 (t,x,y;z)\bigr)\Phi_1^{k} (t,x,y;z) \Phi_2^{l-1} (t,x,y;z) - (\partial_{z_j z_j} \Phi_i (t,x,y;z)\bigr)\Bigr)\Bigr)\Bigr|_{z=(1,1)} \notag\\&= (\mathcal L_i m_{ij}^{(2!)} (t,\cdot,y))(x) - \mu_i m_{ij}^{(2!)} (t,x,y) + \sum_{k+l\geqslant2} \beta_i (k,l) \Bigl( k(k-1)\bigl(m_{1j}^{(1)} (t,x,y)\bigr)^2\notag\\&\quad + k m_{1j}^{(2!)} (t,x,y) + 2kl m_{1j}^{(1)} (t,x,y) m_{2j}^{(1)} (t,x,y) + l(l-1)\bigl(m_{2j}^{(1)} (t,x,y)\bigr)^2\notag\\&\quad + lm_{2j}^{(2!)} (t,x,y) - m_{ij}^{(2!)} (t,x,y) \Bigr).\label{second_factorial_moment_right_hand_part}
\end{align}

By equating Equation~\eqref{second_factorial_moment_left_hand_part} with~\eqref{second_factorial_moment_right_hand_part} we get

\begin{align}\notag
\frac{\partial m_{ij}^{(2!)} (t,x,y)}{\partial t} &= (\mathcal L_i m_{ij}^{(2!)} (t,\cdot,y))(x) - \mu_i m_{ij}^{(2!)} (t,x,y) + \sum_{k+l\geqslant2} \beta_i (k,l)\Bigl( km_{ij}^{(2)} (t,x,y) \\\notag&\quad+ k(k-1)[m_{1j}^{(1)} (t,x,y)]^2 + lm_{2j}^{(2)} (t,x,y) + l(l-1)[m_{2j}^{(1)} (t,x,y)]^2 \\\notag&\quad+ 2klm_{1j}^{(1)} (t,x,y) m_{2j}^{(1)} (t,x,y) - km_{1j}^{(1)} (t,x,y) - lm_{2j}^{(1)} (t,x,y)\Bigr) \\\label{second_factorial_moment_equation}&\quad- \sum_{k+l\geqslant2} \beta_i (k,l) m_{ij}^{(2!)} (t,x,y);\\
\label{second_factorial_moment_initcond} m_{ij}^{(2!)} (0,x,y) &\equiv 0.
\end{align}

Finally, to obtain the differential equations for the second moments $m_{ij}^{(2)} (t,x,y)$, we add the term $\partial_t m_{ij}^{(1)} (t,x,y)$ to each side of Equation~\eqref{second_factorial_moment_equation}. Then, we substitute the term $m_{ij}^{(1)} (t,x,y)$ on the right side of the obtained expression by its representation~\eqref{El21-1}. Then, the left side of the resulting equation takes the form $\partial_t m_{ij}^{(2)} (t,x,y)$, while the right side is equal~to
\begin{align*}
&(\mathcal L_i m_{ij}^{(2!)} (t,\cdot,y))(x) - \mu_i m_{ij}^{(2!)} (t,x,y) + \\&\quad\sum_{k+l\geqslant2} \beta_i (k,l) \Bigl( k(k-1)\bigl(m_{1j}^{(1)} (t,x,y)\bigr)^2 + k m_{1j}^{(2!)} (t,x,y)+ 2kl m_{1j}^{(1)} (t,x,y) m_{2j}^{(1)} (t,x,y) \\&\quad+ l(l-1)\bigl(m_{2j}^{(1)} (t,x,y)\bigr)^2 + lm_{2j}^{(2!)} (t,x,y) - m_{ij}^{(2!)} (t,x,y) \Bigr)+ (\mathcal L_i m_{ij}^{(1)} (t,\cdot,y))(x) \\&\quad - \mu_i m_{ij}^{(1)} (t,x,y) + \sum_{k+l\geqslant2} \beta_i (k,l) (km_{1j}^{(1)} (t,x,y) + lm_{2j}^{(1)} (t,x,y)- m_{ij}^{(1)} (t,x,y)) \\&\quad = (\mathcal L_i m_{ij}^{(2)} (t,\cdot,y))(x) - \mu_i m_{ij}^{(2)} (t,x,y) + \sum_{k+l\geqslant2} \beta_i (k,l) \Bigl( km_{1j}^{(2)} (t,x,y)+ lm_{2j}^{(2)} (t,x,y) \\&\quad + k m_{1j}^{(2!)} (t,x,y) + 2kl m_{1j}^{(1)} (t,x,y) m_{2j}^{(1)} (t,x,y) + l(l-1)\bigl(m_{2j}^{(1)} (t,x,y)\bigr)^2 - m_{ij}^{(2)} (t,x,y) \Bigr)
\end{align*}

Thus, we have proved the following lemma.

\begin{lemma}\label{lemma_second_moments}
Let condition~\eqref{int_ineq} hold. Then, the functions $m_{ij}^{(2)} (t,x,y)$, $i,j=1,2$, satisfy the differential equations
\begin{align}\notag
\frac{\partial m_{ij}^{(2)} (t,x,y)}{\partial t} &= (\mathcal L_i m_{ij}^{(2)} (t,\cdot,y))(x) - \mu_i m_{ij}^{(2)} (t,x,y) + \sum_{k+l\geqslant2} \beta_i (k,l) \Bigl(km_{1j}^{(2)} (t,x,y) \\\notag&\quad+ k(k-1)[m_{1j}^{(1)} (t,x,y)]^2 + lm_{2j}^{(2)} (t,x,y) + l(l-1)[m_{2j}^{(1)} (t,x,y)]^2 \\\label{sec_mom_diff_eq}&\quad+ 2klm_{1j}^{(1)} (t,x,y) m_{2j}^{(1)} (t,x,y)\Bigr) - \sum_{k+l\geqslant2} \beta_i (k,l) m_{ij}^{(2)} (t,x,y)
\end{align}
with the initial condition
\begin{equation}
\label{sec_mom_init_cond}m_{ij}^{(2)} (0,x,y) = \delta_i (j) \delta_x (y).
\end{equation}
\end{lemma}

\begin{remark}\label{R:moments2}
Similar considerations as in Remark~\ref{R:moments1} show that in this case Equation~\eqref{sec_mom_diff_eq} can be treated as a linear differential equation in a Banach space whose right-hand side (for each $t$ and $y$) is a linear bounded operator acting in any of the spaces $l_{p}(\mathbb{Z}^{d})$, $p\ge 1$. Therefore, for the same reasons as in Remark~\ref{R:moments1}, we obtain that $m_{ij}^{(2)} (t,x,y)$ (for each $t$ and $y $) as a function of the variable $x$ belongs to each of the spaces $l_{p}(\mathbb{Z}^{d})$, $p\ge 1$, and is thus bounded.
\end{remark}

So, we have obtained the differential equations for the second moments $m_{ij}^{(2)} (t,x,y)$. In the next section, we will find the solutions for the obtained equations.

\subsection{Solutions of Differential Equations for the Second Moments}

In this section, (as in Section~\ref{solution_first_moments}), we will consider the equations for $m_{ij}^{(2)} (t,x,y)$, $i,j=1,2$, which we obtained in Lemma~\ref{lemma_second_moments}, explicitly in terms of the Fourier transform.
To do this, let us apply the Fourier transform~\eqref{fourier_transform} to the pair of functions $(m_{1j}^{(2)} (t,x,y),m_{2j}^{(2)} (t,x,y))$, $j=1,2$. Then, using the notation~\eqref{a}--\eqref{d} from Section~\ref{the_first_moments}, we obtain
\begin{align}\label{sec_mom_m_1j_fourier}
\frac{\partial\widehat{m}_{1j}^{(2)} (t,\theta,y)}{\partial t} &= a(\theta) \widehat{m}_{1j}^{(2)} (t,\theta,y) + b m_{2j}^{(2)} (t,\theta,y) + f_1^{(j)} (t,\theta,y),\\
\label{sec_mom_m_2j_fourier}
\frac{\partial\widehat{m}_{2j}^{(2)} (t,\theta,y)}{\partial t} &= c \widehat{m}_{1j}^{(2)} (t,\theta,y) + d(\theta) m_{2j}^{(2)} (t,\theta,y) + f_2^{(j)}(t,\theta,y),
\end{align}
where
\begin{equation}\label{sec_mom_m_1j_m_2j_fourier_init_cond}
\widehat{m}_{ij}^{(2)} (0,\theta,y) = \delta_i (j) e^{i (\theta,y)},\qquad  i=1,2,
\end{equation}
and
\begin{align}\notag
f_i^{(j)} (t,\theta,y) &= \sum_{k+l\geqslant2} \beta_i (k,l) \biggl[k(k-1)\Bigl(\widehat{m}_{1j}^{(1)} * \widehat{m}_{1j}^{(1)}\Bigr)(t,\theta,y) \\\label{f-def}&\quad+ 2kl\Bigl(\widehat{m}_{1j}^{(1)} * \widehat{m}_{2j}^{(1)}\Bigr)(t,\theta,y) + l(l-1)\Bigl(\widehat{m}_{2j}^{(1)} * \widehat{m}_{2j}^{(1)}\Bigr)(t,\theta,y)\biggr].
\end{align}

Here,
\[
(F * G)(t,\theta,y) = \biggl(\frac{1}{2\pi}\biggr)^d \int_{[-\pi,\pi]^d} F(t,\theta-v,y)G(t,v,y)\,dv,
\]
i.e., $(F * G)(t,\theta,y)$ is the convolution of the functions $F(t,\theta,y)$ and $G(t,\theta,y)$ with respect to the variable $\theta$.

In what follows we will need the explicit form of the solution of the following linear differential equation:
\[
\frac{dx(t)}{dt} = k x(t) + f(t).
\]

This solution can be readily obtained by the method of \emph{variation of parameters}, also known as the method of \emph{variation of constants}:
\begin{equation}\label{linear_ODE}
x(t)=e^{kt}x(0)+\int_{0}^{t}e^{k(t-s)}f(s)\,ds =
e^{kt} \biggl(x(0)+ \int_0^t f(s)e^{-ks}\,ds \biggr).
\end{equation}

\paragraph{Case $b=c=0$.} Here, the functions $f_i^{(j)} (t,\theta,y)$, $i,j=1,2$, in Equations~\eqref{sec_mom_m_1j_fourier} are identically zero, i.e.,
\[
f_i^{(j)} (t,\theta,y) \equiv 0,~\text{for all}~i,j=1,2,
\]
while Equations~\eqref{sec_mom_m_1j_fourier} split into two independent homogeneous equations
\[
\frac{\partial\widehat{m}_{1j}^{(2)} (t,\theta,y)}{\partial t} = a(\theta)\widehat m_{1j}^{(2)} (t,\theta,y),\quad \frac{\partial\widehat{m}_{2j}^{(2)} (t,\theta,y)}{\partial t} = d(\theta)\widehat m_{2j}^{(2)} (t,\theta,y)
\]
with the initial conditions given by Equation~\eqref{sec_mom_m_1j_m_2j_fourier_init_cond}. Then, applying the formula~\eqref{linear_ODE} we can find the solutions of equations Equations~\eqref{sec_mom_m_1j_fourier}:
\begin{align*}
\widehat{m}_{11}^{(2)} (t,\theta,y) &= e^{i (\theta,y)} e^{a(\theta) t};\\
\widehat{m}_{21}^{(2)} (t,\theta,y) &= 0;\\
\widehat{m}_{12}^{(2)} (t,\theta,y) &= 0;\\
\widehat{m}_{22}^{(2)} (t,\theta,y) &= e^{i (\theta,y)} e^{d(\theta) t}.
\end{align*}

\paragraph{Case $b=0,c>0$.} Here the functions $f_1^{(j)} (t,\theta,y)$, $j=1,2$, in Equations~\eqref{sec_mom_m_1j_fourier} are identically zero, i.e.,
\[
f_1^{(j)} (t,\theta,y) \equiv 0,~\text{for}~j=1,2.
\]

\begin{align*}
\frac{\partial\widehat{m}_{1j}^{(2)} (t,\theta,y)}{\partial t} &= a(\theta) \widehat{m}_{1j}^{(2)} (t,\theta,y);\\
\frac{\partial\widehat{m}_{2j}^{(2)} (t,\theta,y)}{\partial t} &= c \widehat{m}_{1j}^{(2)} (t,\theta,y) + d(\theta) \widehat{m}_{2j}^{(2)} (t,\theta,y) + f_2^{(j)}(t,\theta,y).
\end{align*}

The solution of the first equation due to Formula~\eqref{linear_ODE} is clearly as follows:
\[
\widehat{m}_{1j}^{(2)} (t,\theta,y) = \widehat{m}_{1j}^{(2)} (0,\theta,y) e^{a(\theta)t} = \delta_1 (j) e^{i (\theta,y)} e^{a(\theta)t},
\]
where the first equality follows from Equation~\eqref{linear_ODE}, whereas the second equality follows from Equation~\eqref{sec_mom_m_1j_m_2j_fourier_init_cond}.

To solve the second equation, we again use Formula~\eqref{linear_ODE} assuming $x(t) = \widehat{m}_{2j}^{(2)} (t,\theta,y)$, $k = d(\theta)$ and $f(t) = c \widehat{m}_{1j}^{(2)} (t,\theta,y) + f_2^{(j)}(t,\theta,y)$. Then,
\[
\widehat{m}_{2j}^{(2)} (t,\theta,y) = e^{d(\theta)t} \biggl(\widehat{m}_{2j}^{(2)} (0,\theta,y)+ \int_0^t \Bigl(c \widehat{m}_{1j}^{(2)} (s,\theta,y) + f_2^{(j)}(s,\theta,y)\Bigr)e^{-d(\theta)s}\,ds \biggr),
\]
where by Equation~\eqref{sec_mom_m_1j_m_2j_fourier_init_cond} we have
\begin{align*}
\widehat{m}_{11}^{(2)} (0,\theta,y) &= e^{i (\theta,y)};&\widehat{m}_{21}^{(2)} (0,\theta,y) &= 0&\text{for}~j=1,\\
\widehat{m}_{12}^{(2)} (0,\theta,y) &= 0;&\widehat{m}_{22}^{(2)} (0,\theta,y) &= e^{i (\theta,y)}&\text{for}~j=2.
\end{align*}
Therefore, finally
\begin{align*}
\widehat{m}_{11}^{(2)} (t,\theta,y) &= e^{i (\theta,y)}e^{a(\theta)t};\\
\widehat{m}_{21}^{(2)} (t,\theta,y) &= e^{d(\theta)t} \biggl(\int_0^t e^{-d(\theta)s}\Bigl(ce^{i (\theta,y)}e^{a(\theta)s} + f_2^{(1)} (t,\theta,y)\Bigr)\,ds\biggr);\\
\widehat{m}_{12}^{(2)} (t,\theta,y) &= 0;\\
\widehat{m}_{22}^{(2)} (t,\theta,y) &= e^{d(\theta)t} \biggl(\int_0^t e^{-d(\theta)s}\Bigl(f_2^{(2)} (t,\theta,y)\Bigr)\,ds + e^{i (\theta,y)} \biggr).
\end{align*}

\paragraph{Case $b>0,c=0$.} Similarly to the previous case, here the functions $f_2^{(j)} (t,\theta,y)$, $j=1,2$, in Equations~\eqref{sec_mom_m_1j_fourier} are identically zero, i.e.,
\[
f_2^{(j)} (t,\theta,y) \equiv 0,~\text{for}~j=1,2.
\]
Then, Equations~\eqref{sec_mom_m_1j_fourier} also take the ``triangle'' form
\begin{align*}
\frac{\partial\widehat{m}_{1j}^{(2)} (t,\theta,y)}{\partial t} &= a(\theta) \widehat{m}_{1j}^{(2)} (t,\theta,y) + b\widehat{m}_{2j}^{(2)} (t,\theta,y) + f_1^{(j)}(t,\theta,y);\\
\frac{\partial\widehat{m}_{2j}^{(2)} (t,\theta,y)}{\partial t} &= d(\theta) \widehat{m}_{2j}^{(2)} (t,\theta,y).
\end{align*}

The solution of the second equation is equal to
\[
\widehat{m}_{2j}^{(2)} (t,\theta,y) = \widehat{m}_{2j}^{(2)} (0,\theta,y) e^{d(\theta)t} = \delta_2 (j) e^{i (\theta,y)} e^{d(\theta)t}.
\]
where again the first equality follows from Equation~\eqref{linear_ODE} whereas the second equality follows from Equation~\eqref{sec_mom_m_1j_m_2j_fourier_init_cond}.

To solve the first equation we apply the formula~\eqref{linear_ODE} with $x(t) = \widehat{m}_{1j}^{(2)} (t,\theta,y)$, $k = a(\theta)$ and $f(t) = b\widehat{m}_{2j}^{(2)} (t,\theta,y) + f_1^{(j)}(t,\theta,y)$. Then,
\[
\widehat{m}_{1j}^{(2)} (t,\theta,y) = e^{a(\theta)t} \biggl(\widehat{m}_{1j}^{(2)} (0,\theta,y)+ \int_0^t \Bigl(b\widehat{m}_{2j}^{(2)} (s,\theta,y) + f_1^{(j)}(s,\theta,y)\Bigr)e^{-a(\theta)s}\,ds \biggr).
\]
where by Equation~\eqref{sec_mom_m_1j_m_2j_fourier_init_cond} we have
\begin{align*}
\widehat{m}_{11}^{(2)} (0,\theta,y) &= e^{i (\theta,y)}; &\widehat{m}_{21}^{(2)} (0,\theta,y) &=  0&\text{for}~j=1,\\
\widehat{m}_{12}^{(2)} (0,\theta,y) &=  0; &\widehat{m}_{22}^{(2)} (0,\theta,y) &=  e^{i (\theta,y)}&\text{for}~j=2.
\end{align*}
Therefore,
\begin{align*}
\widehat{m}_{11}^{(2)} (t,\theta,y) &= e^{a(\theta)t} \biggl(\int_0^t e^{-a(\theta)s}\Bigl(f_1^{(j)} (t,\theta,y)\Bigr)\,ds + e^{i (\theta,y)} \biggr);\\
\widehat{m}_{21}^{(2)} (t,\theta,y) &= 0;\\
\widehat{m}_{12}^{(2)} (t,\theta,y) &= e^{a(\theta)t} \biggl(\int_0^t e^{-a(\theta)s}\Bigl(be^{i (\theta,y)}e^{d(\theta)s} + f_1^{(j)} (t,\theta,y)\Bigr)\,ds\biggr);\\
\widehat{m}_{22}^{(2)} (t,\theta,y) &= e^{i (\theta,y)}e^{d(\theta)t}.
\end{align*}

\paragraph{Case $b>0,c>0$.}
To address this case, we first recall the explicit form of the solution to the following linear differential equation:
\begin{equation}\label{linear_system_diff_eq}
\frac{dx(t)}{dt} = A x(t) + f(t),
\end{equation}
where $A$ is a matrix (in our problem $A$ is a $2\times2$ matrix) with time-independent (constant) entries and $f(t)$ is a column-vector function.

The solution of Equation~\eqref{linear_system_diff_eq} can be easily obtained by the method of \emph{variation of parameters}, see, e.g.,~\cite{Filipp} or any other textbook on linear differential equations:
\begin{equation}\label{ODE_system}
x(t)=U(t)x(0)+\int_{0}^{t}U(t-s)f(s)\,ds,
\end{equation}
where the matrix-function $U(t)$ is the so-called ``fundamental solution'' of Equation~\eqref{linear_system_diff_eq}.
Let $U(t) = \exp\{At\}$. Then

It is known~\cite{Filipp} that $U(t)$ can be expressed as $U(t) = \exp\{At\}$. However, for us, the following representation for  $U(t)$ will be more useful:
\[
U(t)=\left(\begin{array}{cc}u_{11}(t)&u_{12}(t)\\u_{21}(t)&u_{22}(t)\end{array}\right),
\]
where the vector-functions
\[
u_{1}(t)=\left(\begin{array}{c}u_{11}(t)\\u_{21}(t)\end{array}\right),\quad
u_{2}(t)=\left(\begin{array}{c}u_{12}(t)\\u_{22}(t)\end{array}\right)
\]
are solutions of the homogeneous system
\[
\frac{dx(t)}{dt} = A x(t),
\]
satisfying the initial conditions, respectively,
\[
u_{1}(0)=\left(\begin{array}{c}1\\0\end{array}\right),\quad
u_{2}(0)=\left(\begin{array}{c}0\\1\end{array}\right).
\]

The components $u_{ij}(t)$ of the solutions $u_{1}(t)$ and $u_{2}(t)$ can be computed by using calculations from Remark~\ref{Yrem_ODE} (case $b>0, c>0$). By doing the needed computations, we obtain:
\begin{align*}
u_{11} (t) &= \frac{1}{\lambda_1 (\theta) - \lambda_2 (\theta)} \biggl((a (\theta) - \lambda_2 (\theta)) e^{\lambda_1(\theta) t} + (\lambda_1 (\theta) - a(\theta)) e^{\lambda_2(\theta) t} \biggr)\\
u_{21} (t) &= \frac{c}{\lambda_1 (\theta) - \lambda_2 (\theta)} \biggl(e^{\lambda_1(\theta) t} - e^{\lambda_2(\theta) t}\biggr)\\
u_{12} (t) &= \frac{b}{\lambda_1 (\theta) - \lambda_2 (\theta)} \biggl(e^{\lambda_1(\theta) t} - e^{\lambda_2(\theta) t}\biggr)\\
u_{22} (t) &= \frac{1}{\lambda_1 (\theta) - \lambda_2 (\theta)} \biggl((\lambda_1(\theta) - a(\theta)) e^{\lambda_1(\theta) t} + (\lambda_2 (\theta)- a(\theta)) e^{\lambda_2(\theta) t} \biggr),
\end{align*}
where $\lambda_{1,2} (\theta)$ are specified by Equation~\eqref{lambda}.

Therefore, from Equations~\eqref{sec_mom_m_1j_m_2j_fourier_init_cond} and~\eqref{ODE_system} we obtain the following solutions for Equations~\eqref{sec_mom_m_1j_fourier} and \eqref{sec_mom_m_2j_fourier}:
\begin{align*}
\widehat m_{11}^{(2)} (t,\theta,y) &= \frac{1}{\lambda_1 (\theta) - \lambda_2 (\theta)}\int_0^t \biggl( (a(\theta)-\lambda_2 (\theta))f_1^{(1)} (s,\theta,y) + bf_2^{(1)} (s,\theta,y) \biggr) e^{\lambda_1(\theta) (t-s)}\,ds \\&+ \frac{1}{\lambda_1 (\theta) - \lambda_2 (\theta)}\int_0^t \biggl( (\lambda_1 (\theta) - a(\theta))f_1^{(1)} (s,\theta,y) - bf_2^{(1)} (s,\theta,y) \biggr) e^{\lambda_2(\theta) (t-s)}\,ds \\&+ \frac{1}{\lambda_1 (\theta) - \lambda_2 (\theta)} \biggl((a(\theta) - \lambda_2(\theta)) e^{\lambda_1(\theta) t} + (\lambda_1(\theta) - a(\theta)) e^{\lambda_2(\theta) t} \biggr) e^{i (\theta,y)},
\end{align*}


\begin{align*}
\widehat m_{21}^{(2)} (t,\theta,y) &= \frac{1}{\lambda_1 (\theta) - \lambda_2 (\theta)}\int_0^t \biggl( cf_1^{(1)} (s,\theta,y) + (\lambda_1(\theta) - a(\theta))f_2^{(1)} (s,\theta,y) \biggr) e^{\lambda_1(\theta) (t-s)}\,ds \\&+ \frac{1}{\lambda_1 (\theta) - \lambda_2 (\theta)}\int_0^t \biggl( -cf_1^{(1)} (s,\theta,y) + (\lambda_2 (\theta) - a(\theta)) f_2^{(1)} (s,\theta,y) \biggr) e^{\lambda_2(\theta) (t-s)}\,ds \\&+ \frac{c}{\lambda_1 (\theta) - \lambda_2 (\theta)} \biggl( e^{\lambda_1 (\theta)t} - e^{\lambda_2 (\theta)t} \biggr) e^{i (\theta,y)},
\end{align*}

\begin{align*}
\widehat m_{12}^{(2)} (t,\theta,y) &= \frac{1}{\lambda_1 (\theta) - \lambda_2 (\theta)}\int_0^t \biggl( (a(\theta) - \lambda_2(\theta))f_1^{(2)} (s,\theta,y) + bf_2^{(2)} (s,\theta,y) \biggr) e^{\lambda_1(\theta) (t-s)}\,ds \\&+ \frac{1}{\lambda_1 (\theta) - \lambda_2 (\theta)}\int_0^t \biggl( (\lambda_1(\theta) - a(\theta)) f_1^{(2)} (s,\theta,y) - bf_2^{(1)} (s,\theta,y) \biggr) e^{\lambda_2(\theta) (t-s)}\,ds \\&+ \frac{b}{\lambda_1 (\theta) - \lambda_2 (\theta)} \biggl( e^{\lambda_1 (\theta)t} - e^{\lambda_2 (\theta)t} \biggr) e^{i (\theta,y)},
\end{align*}
\begin{align*}
\widehat m_{22}^{(2)} (t,\theta,y) &= \frac{1}{\lambda_1 (\theta) - \lambda_2 (\theta)}\int_0^t \biggl( cf_1^{(2)} (s,\theta,y) + (\lambda_2 (\theta)-a(\theta))f_2^{(2)} (s,\theta,y) \biggr) e^{\lambda_1(\theta) (t-s)}\,ds \\&+ \frac{1}{\lambda_1 (\theta) - \lambda_2 (\theta)}\int_0^t \biggl( -cf_1^{(1)} (s,\theta,y) + (\lambda_2 (\theta) - a(\theta))f_2^{(1)} (s,\theta,y) \biggr) e^{\lambda_2(\theta) (t-s)}\,ds \\&+ \frac{1}{\lambda_1 (\theta) - \lambda_2 (\theta)} \biggl((\lambda_1(\theta) - a(\theta)) e^{\lambda_1(\theta) t} + (\lambda_2(\theta) - a(\theta)) e^{\lambda_2(\theta) t} \biggr) e^{i (\theta,y)},
\end{align*}
where the functions $f_{i}^{(j)}(s,\theta,y)$ are defined by Equation~\eqref{f-def}.

\section{Clustering for BRWs with Two Types of Particles with a Critical Reproduction Law}\label{Cluster}

In this section, we consider BRWs with two types of particles satisfying the condition that the particle reproduction law at each lattice point is described by an irreducible critical two-type branching process and that the underlying random walks have finite variances of the jumps. We show that for particles of both types with the underlying recurrent random walks on $\mathbb Z^d$, a phenomenon of clustering of particles can be observed over long times, implying that the majority of particles are concentrated in some particular areas. We generalize the study started in~\cite{BMY:21} for BRW with one type of particles.

\subsection{Degeneration Probability}\label{S:DP}
In this section, based on the results for two-type critical branching processes we show that the probability of degeneracy of the subpopulation tends to $1$ for the underlying recurrent random walk. We also show that, at the same time, subpopulations that are not degenerate exhibit linear growth in $t$ at infinity.

Let us introduce some notation. Denote by $D = (d_{ij})$ the matrix with the elements
\[
d_{ij} := \frac{\partial (F_{i}(z_1, z_2)+\delta_i (1)\beta_1(1,0) z_1 + \delta_i (2)\beta_2(0,1)z_2)}{\partial z_j} \bigg|_{z = (1,1)},\qquad i,j=1,2,
\]
where $F_i(z_1,z_2)$ are the generating functions defined in Equation~\eqref{gen_function}.
We also define the densities of second factorial moments of $F_i (z_1,z_2)$ (cf. Equation~(4) in~\cite{Sevast}(Ch.~4, \S~7)) as
\[
b^{(i)}_{jk} := \frac{\partial^2 F_{i}(z_1, z_2)}{\partial z_j \partial z_k}\bigg|_{z = (1,1)},\qquad i,j,k = 1,2,
\]
and assume that condition~\eqref{int_ineq} holds, so that $d_{ij}$ and $b^{(i)}_{jk}$ are finite for all $i,j,k=1,2$.

Recall the following definition from~\cite{Sevast} (Ch.~4, \S~5, Def.~2):
\begin{definition}
A matrix $C = (c_{ij})$, $i,j=1,\ldots,n$ is called \emph{reducible} if there exist two subsets $S_1,S_2\in\{1,\ldots,n\}$ and $S_1\bigcap S_2 = \emptyset$ such that $c_{ij} = 0$ for all $i\in S_1$, $j\in S_2$. Otherwise, the matrix $C$ is called \emph{irreducible}.
\end{definition}

\begin{definition}
A branching process is called \emph{irreducible}~\cite{Sevast}(Ch.~4, \S~6, Th.~2) if the matrix $D$ is~irreducible.
\end{definition}

Now, note that due to~\eqref{E-mij}, $m_{ij}^{(1)} (t,x,y)\equiv m_{ij}^{(1)} (t,x-y,0)$ and then
\begin{equation}\label{sum-m1txy}
\sum_{y \in \mathbb{Z}^d} m_{ij}^{(1)} (t,x,y)= \sum_{y \in \mathbb{Z}^d} m_{ij}^{(1)} (t,x-y,0)=\sum_{z \in \mathbb{Z}^d} m_{ij}^{(1)} (t,z,0),\qquad i,j = 1,2,
\end{equation}
where the sum on the right-hand side is finite due to Remark~\ref{R:moments1}. Then, the matrix $D(t, x) := (d_{ij}(t,x))$ with elements
\[
d_{ij}(t, x) := \Expect \sum_{y \in \mathbb{Z}^d} n_{ij}(t, x, y) = \sum_{y \in \mathbb{Z}^d} m_{ij}^{(1)} (t,x,y),\qquad i,j = 1,2,
\]
is well-defined, i.e., its elements $d_{ij}(t, x)$ are finite. Moreover, the relations~\eqref{sum-m1txy} prove that the quantity $d_{ij}(t, x)$ does not indeed depend on the spatial coordinate $x$, i.e.,
\begin{equation}\label{dij}
d_{ij}(t, x)=d_{ij}(t)~\text{for all}~x\in\mathbb Z^d.
\end{equation}

Then, according to~\cite{Sevast}(Ch.~4, \S~7, Th.~5), we have
\begin{equation}\label{d_ij(t)}
d_{ij}(t)= u_i v_j e^{rt} + o(e^{r_1 t})\qquad\text{for}\quad t\to\infty, \qquad i,j = 1,2,
\end{equation}
where $r$ is the Perron root (see~\cite{Sevast}(Ch.~4, \S~5, Def.~6)) of the matrix $D$ and $r_1$ is some quantity satisfying $r_1<r$. We denote by
\[
u=(u_1,u_2),\qquad v=(v_1,v_2)
\]
the left and right eigenvectors,  respectively, corresponding to the eigenvalue $r$ of $D$.

\begin{definition}
An irreducible branching process is called \emph{critical}~\cite{Sevast}(Ch.~4, \S~7, Def.~2) if $r = 0$ and
\[
\sum_{i=1}^2 \sum_{j=1}^2 \sum_{k=1}^2 v_i b^{(i)}_{jk} u_j u_k > 0.
\]
\end{definition}

Let
\begin{equation}\label{n_i(t,x)}
n_i(t, x) = \sum_{j=1}^2 \sum_{y \in \mathbb{Z}^d}n_{ij}(t, x, y)
\end{equation}
be the number of particles in a subpopulation at time $t$ generated by a particle of the $i$-th type provided that at the initial moment of time the particle was at the point $x$.

\begin{remark}
Evaluate the quantity $n_i (t,x)$, $i=1,2$, at the time moment $t+dt$. Let $n_{ij} (t,x)$, $j=1,2$ be the number of offsprings of type $j$ generated by a single particle of type $i$, which at the time moment $t=0$ was located at the point $x\in\mathbb Z^d$, so that $n_i (t,x) = n_{i1} (t,x) + n_{i2} (t,x)$. Assume $G_i (t,x;z) = \Expect z^{n_i (t,x)}$. Then, by using the Kolmogorov forward equation, we obtain the following relations:
\begin{align*}
G_i (t+dt,x;z) &= \Expect z^{n_i (t+dt,x)}\\&= \Expect z^{n_i (t+dt,x)}\Bigl[ \sum_{k+l\geqslant2} \bigl(\beta_1 (k,l)n_{i1} (t,x) + \beta_2 (k,l)n_{i2} (t,x) \bigr)z^{k+l}dt \\&\quad+ \bigl(\mu_1 n_{i1} (t,x) + \mu_2 n_{i2} (t,x) \bigr)z^{-1}\,dt + \bigl(\varkappa_1 n_{i1} (t,x) + \varkappa_2 n_{i2} (t,x) \bigr)\,dt \\&\quad+ \bigl( 1 - \varkappa_1 n_{i1} (t,x)\,dt - \varkappa_2 n_{i2} (t,x)\,dt - \sum_{k+l\geqslant2} \bigl(\beta_1 (k,l)n_{i1} (t,x) \\&\quad+ \beta_2 (k,l)n_{i2} (t,x) \bigr)z^{k+l}\,dt \bigr) - \mu_1 n_{i1} (t,x)\,dt - \mu_2 n_{i2} (t,x)\,dt + o(dt) \Bigr] \\&= \Expect z^{n_i (t+dt,x)}\Bigl[ \sum_{k+l\geqslant2} \bigl(\beta_1 (k,l)n_{i1} (t,x) + \beta_2 (k,l)n_{i2} (t,x) \bigr)z^{k+l}\,dt \\&\quad+ \bigl(\mu_1 n_{i1} (t,x) + \mu_2 n_{i2} (t,x) \bigr)z^{-1}\,dt + \bigl( 1 - \sum_{k+l\geqslant2} \bigl(\beta_1 (k,l)n_{i1} (t,x) \\&\quad+ \beta_2 (k,l)n_{i2} (t,x) \bigr)z^{k+l}\,dt \bigr) - \mu_1 n_{i1} (t,x)\,dt - \mu_2 n_{i2} (t,x)\,dt + o(dt) \Bigr]
\end{align*}

From these relations, it can be seen that the behavior of the process $n_i (t,x)$ depends only on its ``branching component'' and the evolution of the process coincides with the evolution of the branching process with continuous time treated in~\cite{Sevast}. For this reason, we apply the results of~\cite{Sevast} in the following.
\end{remark}

\begin{remark}
Note that from Remark~\ref{R-uniform}, we have for all $k\in\mathbb Z_+$:
\begin{align*}
\Prob(n_i(t, x)=k) &= \Prob\Bigl(\sum_{j=1}^2 \sum_{y \in \mathbb{Z}^d}n_{ij}(t, x, y)=k\Bigr) \\&= \Prob \Bigl(\sum_{j=1}^2 \sum_{y \in \mathbb{Z}^d}n_{ij}(t, 0, y-x)=k\Bigr) = \Prob (n_i (t,0)=k).
\end{align*}
\end{remark}

Recall that the branching process under consideration is assumed to be critical and irreducible. In this case,~\cite{Sevast}(Ch.~6, \S~3, Th.~4) implies that the probability of non-degeneration of a subpopulation has the following asymptotic behavior for all $x\in\mathbb Z^d$ as $t\to\infty$:
\begin{equation}
\label{not deg}
\begin{split}
\Prob\big(n_i(t, x) > 0\big) &= \Prob\big(n_i(t, 0) > 0\big) = \frac{c_i}{t} + o\left(\frac{1}{t}\right) \to 0,\\
\Prob\big(n_i(t, x) = 0\big) &= \Prob\big(n_i(t, 0) = 0\big) = 1 - \frac{c_i}{t} + o\left(\frac{1}{t}\right) \to 1,
\end{split}
\end{equation}
where $c_i$ is a constant. Thus, the \emph{probability of degeneration} $\Prob\big(n_i(t, x) = 0\big)$ of the subpopulation tends to $1$ for all $x\in\mathbb Z^d$ as $t\to\infty$.

Now, we will estimate the conditional mathematical expectation
\[
\Expect \bigg(\sum_{y \in \mathbb{Z}^d} n_{ij}(t, x, y) \Big| n_i(t, x) > 0\bigg)
\]
which is the main object of the study in Section~\ref{S:DP}.
By the definition of conditional expectation we have for all $x\in\mathbb Z^d$
\[
\Expect \bigg(\sum_{y \in \mathbb{Z}^d} n_{ij}(t, x, y) \Big| n_i(t, x) > 0\bigg) = \frac{\Expect \Bigl(\sum_{y \in \mathbb{Z}^d} n_{ij}(t, x, y) \mathbb I\{n_i (t,x)>0\}\Bigr)}{\Prob (n_i (t,x)>0)},
\]
where $\mathbb I \{A\}$ is the indicator of the set $A$. Note that from Equation~\eqref{n_i(t,x)} it follows that
\[
n_i (t,x) = 0 \Longrightarrow \sum_{y\in\mathbb Z^d} n_{ij} (t,x,y) = 0\quad~\text{for}~j=1,2.
\]

Then, with the usage of the formula of total probability we have
\begin{align*}
\Expect \sum_{y \in \mathbb{Z}^d} n_{ij}(t, x, y) &= \Expect \Bigl( \sum_{y \in \mathbb{Z}^d} n_{ij}(t, x, y) \bigl(\mathbb I\{n_i (t,x)>0\} + \mathbb I\{n_i (t,x)=0\}\bigr) \Bigr) \\&= \Expect \sum_{y \in \mathbb{Z}^d} n_{ij}(t, x, y) \mathbb I\{n_i (t,x)>0\} + \Expect \sum_{y \in \mathbb{Z}^d} n_{ij}(t, x, y) \mathbb I\{n_i (t,x)=0\} \\&= \Expect \sum_{y \in \mathbb{Z}^d} n_{ij}(t, x, y) \mathbb I\{n_i (t,x)>0\}.
\end{align*}

Thus, from Equation~\eqref{dij} we obtain
\[
d_{ij}(t) = d_{ij} (t,x) = \Expect \sum_{y \in \mathbb{Z}^d} n_{ij}(t, x, y) = \Expect \bigg(\sum_{y \in \mathbb{Z}^d} n_{ij}(t, x, y) \Big| n_i(t, x) > 0\bigg) P\bigg(n_i(t, x) > 0\bigg).
\]

At the same time, substituting $r = 0$ in Equation~\eqref{d_ij(t)} we obtain $d_{ij} (t) = u_i v_j + o(1)$, whence, denoting $C_{ij} := \frac{u_i v_j}{c_i} = \mathrm{const}$ and using Equation~\eqref{not deg} we get
\begin{equation}
\label{eq: limth 2}
\Expect \bigg(\sum_{y \in \mathbb{Z}^d} n_{ij}(t, x, y) \Big| n_i(t, x) > 0\bigg) = \frac{u_i v_j + o(1)}{c_i/t+o(1/t)} = C_{ij}t + o(t)\quad\text{as}\quad  t \to \infty.
\end{equation}

In virtue of Equation~\eqref{not deg} we have that, in the case when $t\to\infty$, the probability of degeneration of the subpopulation $\Prob\big(n_i(t, x) = 0\big)$ tends to $1$.
At the same time, due to Equation~\eqref{eq: limth 2} those subpopulations that are not degenerate have a linear growth in $t$ at~infinity.

\subsection{Clustering}
In this section, we study the effect of clustering at each point for an irreducible critical branching process under the condition that the tail of a random walk is superexponentially light, i.e., for each $\lambda\in \mathbb{R}^{d}$, $i=1,2$, the following condition holds:
\[
\sum_{v \in \mathbb{Z}^d} e^{(\lambda,v)} a_i(v) < \infty.
\]

By $p_i(t, x, y)$, we denote the transition probability of the random walk on $\mathbb{Z}^d$ defined by $\mathcal{L}_i$, $i = 1, 2$, see Equation~\eqref{generator}.
From \cite{Gikhman}(Ch.~3, \S~2) it follows that $p_i(t, x, y)$ is the solution of the Cauchy problem
\[
\frac{\partial p_i(t,x,y)}{\partial t} = (\mathcal L_i p_i(t,\cdot,y))(x),\qquad
p_i(0,x,y) = \delta_x (y).
\]
Note that here, as was shown in Remark~\ref{remark_p(t,x,y)}, $p_i(t,x,y) = p_i(t,x-y,0) = p_i (t,0,y-x)$, which follows from the property of spatial homogeneity of the process under consideration.

Denote $y-x=s$, then from~\cite{MolchanovYarovaya2013}(Eq.~(4.7)) we have for $s = O(\sqrt{t})$ the following~equality:
\begin{equation}\label{p_i^d(t,0,s)}
p_i(t, 0, s) = \frac{e^{-(B_i^{-1}s, s)/(2t)}}{(2\pi t)^{d/2}\sqrt{\det B_i}}+o(t^{-d/2}),
\end{equation}
where $B_i = \bigl(b_i^{(kj)}\bigr)_{k,j}$, $i = 1,2$, is the matrix with the elements
\[
b_i^{(kj)} = \sum_{v} a_i(v) v^k v^j,\quad v=(v^1,\ldots,v^d),\quad k,j = 1,\ldots,d.
\]

For $d = 1$, the one-dimensional matrix $B_i$ is as follows: $B_i = \bigl( b_i^{(11)} \bigr)$. We denote $b_i := b_i^{(11)}$, then Equation~\eqref{p_i^d(t,0,s)} takes the form
\[
p_i^{(1)}(t, 0, s) = \frac{e^{-s^2/(2 b_i t)}}{\sqrt{2\pi b_i t}} + o(1/\sqrt{t}).
\]

Consider the probability that a particle located at point $0\in\mathbb Z^d$ will jump no further than a distance $C\sqrt{t}$, $C>0$ is some constant. Then,
\begin{align*}
\sum_{|s|<C\sqrt{t}} p_i^{(1)}(t, 0, s) &= p_i^{(1)}(t, 0, 0) + 2\sum_{s \in \mathbb{N}, s < C\sqrt{t} } p_i^{(1)}(t, 0, s) \\&> p_i^{(1)}(t, 0, 0) + 2\int_{1}^{C\sqrt{t}} \frac{e^{-\tau^2/(2 b_i t)}}{\sqrt{2\pi b_i t}}\,d\tau = 2\int_{0}^{C\sqrt{t}} \frac{e^{-\tau^2/(2 b_i t)}}{\sqrt{2\pi b_i t}}\,d\tau + o(1).
\end{align*}

Note that in the last equality, under the sign of integral is the function $f(\tau)$ such that
\[
f(\tau) = \frac{e^{-\tau^2/(2 b_i t)}}{\sqrt{2\pi b_i t}}.
\]

Function $f(\tau)$ is the density function of a random variable, which has normal distribution with mean $0$ and variance $b_i t$. Therefore, choosing appropriate constant $C>0$, we can make the quantity
\[
2\int_{0}^{C\sqrt{t}} \frac{e^{-\tau^2/(2 b_i t)}}{\sqrt{2\pi b_i t}}\,d\tau
\]
to be arbitrarily close to 1. Hence, for every $\varepsilon > 0$ there exists $C >0$ such that
\begin{equation}
\label{eq: limth 3}
\sum_{|y-x|<C\sqrt{t}} p_i^{(1)}(t, x, y) > 1-\varepsilon.
\end{equation}

Thus, as $t\to\infty$ a particle will move at a distance of no more than $C\sqrt{t}$ with probability arbitrarily close to $1$.

Turn to the case when $\mathbb Z^d$, $d=2$. Then, from Equation~\eqref{p_i^d(t,0,s)} takes the form
\[
p_i^{(2)}(t, 0, s) = \frac{e^{-(B_i^{-1}s, s)/(2 t)}}{2\pi t\sqrt{\det B_i}} + o(1/t),
\]

As in the previous case, consider the probability that a particle located at point $0\in\mathbb Z^d$ will jump no further than a distance $C\sqrt{t}$.
Consequently,
\begin{equation}
\label{eq: ineq}
\begin{split}
\sum_{|s|<C\sqrt{t}} p_i^{(2)}(t, 0, s) &= p_i^{(2)}(t, 0, 0) + 4\sum_{s_1 \in \mathbb{N}, s_2 \in \mathbb{Z}_+, \sqrt{s_1^2+s_2^2}<C\sqrt{t} } p_i^{(2)}(t, 0, s) \\
& > p_i^{(2)}(t, 0, 0) + 4 \int_{\tau_1 \geqslant1, \tau_2 \geqslant0, \sqrt{\tau_1^2+\tau_2^2}<C\sqrt{t}} \frac{e^{-(B_i^{-1}\tau, \tau)/(2 t)}}{2\pi t\sqrt{\det B_i}}\,d\tau \\
& > p_i^{(2)}(t, 0, 0) + 4 \int_{\tau_1 \geqslant0, \tau_2 \geqslant0, \sqrt{\tau_1^2+\tau_2^2}<C\sqrt{t}} \frac{e^{-(B_i^{-1}\tau, \tau)/(2 t)}}{2\pi t\sqrt{\det B_i}}\,d\tau \\
& - 4 \int_{0 \leqslant\tau_1 \leqslant1, 0 \leqslant\tau_2 \leqslant C\sqrt{t}} \frac{e^{-(B_i^{-1}\tau, \tau)/(2 t)}}{2\pi t\sqrt{\det B_i}}\,d\tau \\
& = \int_{\sqrt{\tau_1^2+\tau_2^2}<C\sqrt{t}} \frac{e^{-(B_i^{-1}\tau, \tau)/(2 t)}}{2\pi t\sqrt{\det B_i}}\,d\tau + o(1).
\end{split}
\end{equation}

Note that in the last equality, under the sign of integral is the function $f(\tau_1,\tau_2)$ such~that
\[
f(\tau_1, \tau_2) = \frac{e^{-(B_i^{-1}\tau, \tau)/(2 t)}}{2\pi t\sqrt{\det B_i}}
\]

Function $f(\tau_1,\tau_2)$ is the density function of a random variable which has \linebreak two-dimensional normal distribution with mean vector $(0,0)$ and covariance matrix $B_i t$.

Therefore, choosing an appropriate constant $C$ we can get
\[
\int_{\sqrt{\tau_1^2+\tau_2^2}<C\sqrt{t}} \frac{e^{-(B_i^{-1}\tau, \tau)/(2 t)}}{2\pi t\sqrt{\det B_i}}\,d\tau
\]
arbitrarily close to 1. Hence, due to Equation~\eqref{eq: ineq} for every $\varepsilon > 0$, there exists $C >0$, such~that
\[
\sum_{|y-x|<C\sqrt{t}} p_i^{(2)}(t, x, y) > 1-\varepsilon.
\]

Thus, as $t \to \infty$, a particle will move with probability arbitrarily close to $1$ over a distance no greater than $C\sqrt{t}$. This result for the lattice dimension $d=2$ is similar to that obtained in Equation~\eqref{eq: limth 3} for the lattice dimension $d=1$.

Let us now consider the situation where there is one  particle of type $i$ at each point at the initial time. We denote the set of odd positive integers by $\mathbb{N}_1$ and the set of even positive integers by $\mathbb{N}_2$. Let us consider a given particle at time $t$ and all its progenitors up to the initial time.

We consider a particle at time $t$ and a sequence (evolutionary lineage ) $K$ consisting of all its $m$ progenitors (from the initial particle to the immediate parent) and the particle itself, $K = (k_1, \dots, k_m, k_{m+1})$, $m \ge 0$. If $m > 0 $ (i.e., the particle is not included in the set of initial particles on the lattice), then we select $s$ from the sequence of indices $[2, \dots, m+1]$ such that $type(k_s) \ne type(k_{s-1})$, where $ type(k)$ denotes the type of the particle $k$. We denote the sequence of selected indices by $S = (s_1, \dots, s_n)$, $n \le m+1$. If the sequence $S$ turns out to be empty (i.e., no type changes were observed in the evolutionary lineage considered), then we add to it the index $s_1:=m+1$, then $n=1$. We denote by $h(k)$ the lifetime of the particle $k$ and construct the sequence $\tau = (\tau_1, \dots, \tau_n)$, where $\tau_1 = \sum\limits_{i=1}^{S[1]} h(k_i)$ and $\tau_j = \sum\limits_{i=1}^{S[j]} h(k_i) - \tau_{j-1}$, $j=2, \dots, n$. Note that $\sum_{i=1}^n \tau_i = t$. Assuming that the evolutionary lineage started with a particle of type $1$, we obtain that in the time intervals $(0, \tau_1)$, $(\tau_2, \tau_3)$, $\ldots$ the particles of this evolution lineage walk on the lattice under the action $\mathcal{L}_1$ and on the time intervals $(\tau_1, \tau_2)$, $(\tau_3, \tau_4)$, $\dots$ under the action of $\mathcal{L}_2$.

We denote by $p(\tau, x, y)$ the probability for a particle to move from a point $x$ to $y$ on $\mathbb{Z}^d$ in time $\tau$. Due to the Kolmogorov--Chapman equation, for $n\ge2$ we obtain
\begin{equation}\label{p(tau, x, y)}
p(\tau, x, y) = \sum_{\substack{x_i \in \mathbb{Z}^d \\ 1 \leqslant i \leqslant n-1}} \bigg(p_1(\tau_1, x, x_1) \prod_{i=2}^{n-1} p_{s(i)}(\tau_{i}, x_{i-1}, x_{i}) p_{s(n)}(\tau_n, x_{n-1}, y) \bigg),
\end{equation}
where $s(i) = 1$ for $i \in \mathbb{N}_1$ and $s(i) = 2$ for $i \in \mathbb{N}_2$. This representation will be needed in the following lemma.

\begin{lemma}\label{L51}
Let $t_1 := \sum_{i \in \mathbb{N}_1} \tau_i$ be the total time spent by a particle in the first state and $t_2 := \sum_{i \in \mathbb{N}_2} \tau_i$ be the total time spent by the same particle in the second state. Then,
\begin{equation}
\label{lemma_5.1}
p(\tau, x, y) = \sum_{x' \in \mathbb{Z}^d} p_1(t_1, x, x') p_2(t_2, x', y).
\end{equation}
\end{lemma}

\begin{proof}
Let us show first that
\begin{equation}\label{E:ptxy}
p(\tau, x, y) = p((\tau_1, \dots, \tau_{n-2}+\tau_n, \tau_{n-1}), x, y).
\end{equation}

For the proof, due to Equation~\eqref{p(tau, x, y)}, it is enough to consider the following sequence of~relations:
\begin{align*}
p(\tau,& x, y) = \sum_{\substack{x_i \in \mathbb{Z}^d \\ 1 \leqslant i \leqslant n-1}} \bigg(p_1(\tau_1, x, x_1) \prod_{i=2}^{n-3} p_{s(i)}(\tau_{i}, x_{i-1}, x_{i})
p_{s(n-2)}(\tau_{n-2}, x_{n-3}, x_{n-2}) \\
&\quad\times p_{s(n-1)}(\tau_{n-1}, x_{n-2}, x_{n-1})
p_{s(n)}(\tau_n, x_{n-1}, y) \bigg) = \sum_{\substack{x_i \in \mathbb{Z}^d \\ 1 \leqslant i \leqslant n-2}} \sum_{x' \in \mathbb{Z}^d} p_1(\tau_1, x, x_1)\\
&\quad\times \prod_{i=2}^{n-3} p_{s(i)}(\tau_{i}, x_{i-1}, x_{i}) p_{s(n-2)}(\tau_{n-2}, x_{n-3}, x_{n-2}) p_{s(n-1)}(\tau_{n-1}, x', y) p_{s(n)}(\tau_n, x_{n-2}, x') \\
&= \sum_{\substack{x_i \in \mathbb{Z}^d \\ 1 \leqslant i \leqslant n-3}} \sum_{x' \in \mathbb{Z}^d} p_1(\tau_1, x, x_1) \prod_{i=2}^{n-3} p_{s(i)}(\tau_{i}, x_{i-1}, x_{i}) p_{s(n-1)}(\tau_{n-1}, x', y) \\
&\quad\times \sum_{x_{n-2}} p_{s(n-2)}(\tau_{n-2}, x_{n-3}, x_{n-2}) p_{s(n)}(\tau_n, x_{n-2}, x') = \sum_{\substack{x_i \in \mathbb{Z}^d \\ 1 \leqslant i \leqslant n-3}} \sum_{x' \in \mathbb{Z}^d} p_1(\tau_1, x, x_1) \\
&\quad\times \prod_{i=2}^{n-3} p_{s(i)}(\tau_{i}, x_{i-1}, x_{i}) p_{s(n-2)}(\tau_{n-2}+\tau_n, x_{n-3}, x') p_{s(n-1)}(\tau_{n-1}, x', y) \\
&= \sum_{\substack{x_i \in \mathbb{Z}^d \\ 1 \leqslant i \leqslant n-2}} p_1(\tau_1, x, x_1) \prod_{i=2}^{n-3} p_{s(i)}(\tau_{i}, x_{i-1}, x_{i}) p_{s(n-2)}(\tau_{n-2}+\tau_n, x_{n-3}, x_{n-2}) \\
&\quad\times p_{s(n-1)}(\tau_{n-1}, x_{n-2}, y) = p((\tau_1, \dots, \tau_{n-2}+\tau_n, \tau_{n-1}), x, y).
\end{align*}
Consecutively, applying the Formula \eqref{E:ptxy} $n-2$ times, we obtain
\begin{align*}
p(\tau, x, y) &= p((\tau_1,\ldots,\tau_{n-2}+\tau_n, \tau_{n-1}), x, y) \\
&= p((\tau_1, \dots, \tau_{n-3}+\tau_{n-1}, \tau_{n-2}+\tau_n), x, y) = \ldots \\
&= p((\tau_1+\tau_3 +\ldots + \tau_{2[(n-1)/2]+1}, \tau_2+\tau_4 +\ldots + \tau_{2[n/2]}), x, y),
\end{align*}
whence the assertion of the lemma follows.
\end{proof}

Now, we will apply Equation~\eqref{lemma_5.1} from Lemma~\ref{L51} for understanding how far the particles can go from the initial position of their initial progenitor by some time $t$ when $t \to \infty$.

In case of $t \to \infty$ and $t_i \to \infty$, $i=1, 2$, for every $\varepsilon > 0$, there exists $C_i > 0$ such that
\[
\sum_{|y-x|<C_i\sqrt{t}} p_i(t_i, x, y) \geqslant\sum_{|y-x|<C_i\sqrt{t_i}} p_i(t_i, x, y) > 1-\varepsilon.
\]

In case of $t \to \infty$ and $t_i < C$, $i=1, 2$, and for every $\varepsilon >0$, we have:
\[
\sum_{|y-x|<\sqrt{t}} p_i(t_i, x, y) > 1-\varepsilon.
\]

Thus, for $t \to \infty$, $\forall \varepsilon >0$, $\forall \tau: \sum_i \tau_i = t$,
\[
\sum_{|y-x|<(C_1+C_2+1)\sqrt{t}} p(\tau, x, y) = \sum_{|y-x|<(C_1+C_2+1)\sqrt{t}} \sum_{x' \in \mathbb{Z}^d} p_1(t_1, x, x') p_2(t_2, x', y) > (1-\varepsilon)^2.
\]

For $d = 1$, the distance between the start points of subpopulations that did not degenerate by the time $t \to \infty$ has a geometric distribution with an average value of $\frac{t}{c_i} + o(t)$ Equation~\eqref{not deg} and non-degenerate subpopulations have particles at a distance from the initial particle of the order of no more than $\sqrt{t}$ with a probability arbitrarily close to 1, see Equation~\eqref{eq: limth 3}.

Thus, particle clusters with length of order $\sqrt{t}$ are separated by empty intervals with length of order $t$.

Let us turn to the case $d = 2$. Choose two functions, $\nu(t)$ and $f(t)$, such that $\nu(t) \to \infty$ and  $\nu(t)/t \to 0$ for $t \to \infty$, and $f(t)= O\left(\frac{1}{t\nu(t)}\right)$. Now, consider the square of the lattice with side $\sqrt \frac{t\nu(t)e^{\nu(t)}}{c_i}$ and divide it into cells with side $\sqrt{\frac{t\nu(t)}{c_i}}$; then the number of cells will be $e^{\nu(t)}$. We call a cell degenerate at time $t$ if it does not contain the starting points of populations that do not degenerate at time $t$. Then, the probability that for $t \to \infty$ all subpopulations of cell degenerate is
\begin{align*}
\Prob_{deg}(t) = \bigg(1-\frac{c_i}{t} + f(t)\bigg)^{\frac{t\nu(t)}{c_i}} = e^{-\nu(t)+O(1)} \ge \frac{C}{e^{\nu(t)}}
\end{align*}
with some constant $C>0$. The probability of the existence of a cell, whose all subpopulations of the initial particles are degenerated, is
\begin{align*}
1-\bigg(1- \Prob_{deg}(t) \bigg)^{e^{\nu(t)}} \ge 1-\frac{1}{e^C}.
\end{align*}

Non-degenerate subpopulations have particles at a distance from the initial particle of the order of no more than $\sqrt{t} \ll \sqrt{\frac{ t\nu(t)}{c_i}}$.

Therefore, by the time $t \to \infty$, we get particle-free circles with a radius of the order of $\sqrt{\frac{ t\nu(t)}{c_i}}$ at a distance of the order of $\sqrt \frac{t\nu(t)e^{\nu(t)}}{c_i}$.

Thus, we have proved that both in the case of dimension $d = 1$ and $d = 2$, the effect of clustering of particle subpopulations takes place.

\section{Example}
\label{important_example}
One of the assumptions of the model we presented in Section~\ref{description_of_the_model} was the fact that particles cannot change their type over time (see Remark~\ref{remark_mu_beta}). Here, we will consider an example where particles of the first type can become particles of the second type. This example can describe the distribution of a virus.

In Section~\ref{example_description_of_the_model}, we will describe a new model of BRW with two types of particles, where the particles can change their types. We will use the designations from Section~\ref{description_of_the_model}. In Section~\ref{example_the_first_moments}, we study the first moments for the number of particles of type $i=1,2$ at each lattice point. In Section~\ref{example_second_moment_N_1(t,x)}, we obtain the solutions for the second moment for the number of particles of the first type at each lattice point in the more general case. In Section~\ref{example_intermittency_N_1(t,x)}, we study the effect of intermittency in the simplest case for the number of particles of the first type. In Section~\ref{example_second_moment_N_2(t,x)}, we obtain the differential equation for the second moment for the number of particles of the second type at each lattice point and find its asymptotic behavior as $t\to\infty$.

\subsection{Description of the Model}
\label{example_description_of_the_model}

Consider a new model of BRW with two types of particles. Here, we will study the behavior of the processes $N_i(t,x)$, $i=1,2$ defined in~\eqref{total_population}. We call the particles of the first type \emph{infected} and the particles of the second type \emph{particles with immunity}. Let us denote by $r$ the intensity to build up immunity for an infected particle during the small time $dt$. This means that the particle can change type with probability $r\,dt + o(dt)$. Moreover, we assume that there was only one infected particle on the lattice at time $t=0$. Without limiting generality, we can assume that this initial particle was at the origin. Then, $N_1 (0,x) = \delta_0 (x)$, $N_2 (0,x) \equiv 0$ for all $x\in\mathbb Z^d$. Let $b_n$, $n\geqslant2$ be the intensity to infect $n-1$ new particles. Here we assume that there are enough healthy particles at each point of the lattice to get sick. We are also interested in studying the moments of the number of particles of both types. In the previous notation, we say that $b_n = \beta_1 (n,0)$, $\beta_2 (k,l) \equiv 0$ for all $k,l: k+l\geqslant2$.

In what follows, due to the fact that particles can change their types, we use the forward Kolmogorov equations approach to obtain the differential equations for the moments of $N_i (t,x)$, $i=1,2$. The derivation of the forward Kolmogorov equations is based on the following representation:
\begin{align*}
N_1(t+dt,x) &= N_1(t,x) + \xi (dt,x),\\
N_2(t+dt,x) &= N_2(t,x) + \psi (dt,x),
\end{align*}
where $\xi(dt,x)$ and $\psi(dt,x)$ are discrete random variables with the following distributions:
\begin{equation*}
\xi (dt,x) = \begin{cases}
n-1 &\text{with probability}~b_{n} N_1(t,x)\,dt + o(dt),~ n \geqslant 3,\\
1 &\text{with probability}~b_{2} N_1(t,x)\,dt + \varkappa_1 \sum_{z \ne 0} a_1 (-z)N_1(t,x+z)\,dt + o(dt),\\
-1 &\text{with probability}~\mu_1 N_1(t,x)\,dt + \varkappa_1 N_1(t,x)\,dt + rN_1 (t,x)\,dt + o(dt),\\
0 &\text{with probability}~1 - \sum_{n \geqslant 3} b_{n} N_1(t,x)\,dt\\
&- (\beta_{2} + \mu_1 + \varkappa_1)N_1(t,x)\,dt - rN_1 (t,x)\,dt\\
& -\sum_{z \ne 0} a_1 (-z)N_1(t,x+z)\,dt + o(dt);
\end{cases}
\end{equation*}
\begin{equation*}
\psi (dt,x) = \begin{cases}
1 &\text{with probability}~\varkappa_2 \sum_{z \ne 0} a_2 (-z)N_2(t,x+z)\,dt + rN_1(t,x)\,dt + o(dt),\\
-1 &\text{with probability}~\mu_2 N_2(t,x)\,dt + \varkappa_2 N_2(t,x)\,dt + o(dt),\\
0 &\text{with probability}~1 - (\mu_2 + \varkappa_2 N_2(t,x)\,dt - rN_1(t,x)\,dt\\
& -\sum_{z \ne 0} a_2 (-z)N_2(t,x+z)\,dt + o(dt).
\end{cases}
\end{equation*}

We are going to study the first two moments for the random variables $N_i (t,x)$, $i=1,2$. In the next section, we pay attention to the first moments.

\subsection{The First Moments}
\label{example_the_first_moments}

In this section, we consider the first moments for $N_i (t,x)$, $i=1,2$. We obtain the differential equations for them and find their explicit solutions in terms of the Fourier transform~\eqref{fourier_transform}. We also obtain their asymptotic behavior in certain cases. Define the first moments $R_i (t,x) := \Expect N_i(t,x)$. Note that $R_i (0,x) = \delta_1 (i) \delta_0 (x)$. Let $\mathcal F_{\leqslant t}$ be the sigma-algebra of events up to and including $t$. Note that $\xi (dt,x)$ (and $\psi (dt,y)$) and $\mathcal F_{\leqslant t}$ are independent.

Derive the differential equations for these functions:
\begin{align*}
R_1 (t+dt,x) &= \Expect N_1(t+dt,x) = \Expect[N_1(t,x) + \xi(dt,x)] = R_1 (t,x) \\&\quad+ \Expect[\Expect[\xi(dt,x)| \mathcal F_{\leqslant t}]] = R_1 (t,x) + \sum_{n=2}^{\infty} (n-1)b_n R_1 (t,x)\,dt \\&\quad- \mu_1 R_1 (t,x)\, dt + (\mathcal L_1 R_1 (t,\cdot))(x)\,dt - r M_1 (t,x)\,dt + o(dt).
\end{align*}

Let $\beta = \sum_{n=2}^{\infty} (n-1) b_n$. Then, as $dt\to0$ the differential equation for $R_1 (t,x)$ is
\begin{equation}\label{example_first_eq}
\frac{\partial R_1 (t,x)}{\partial t} = (\beta - \mu_1 - r) R_1 (t,x) + (\mathcal L_1 R_1 (t,\cdot))(x),\qquad
R_1 (0,x) = \delta_0 (x).
\end{equation}

The same technique helps to find the differential equation for $R_2 (t,x)$:
\begin{align*}
R_2 (t+dt,x) &= \Expect N_2(t+dt,x) = \Expect[N_2(t,x) + \psi(dt,x)] \\&= R_2 (t,x) + \Expect[\Expect[\psi(dt,x)| \mathcal F_{\leqslant t}]] = R_2 (t,x) \\&\quad+ (\mathcal L_2 R_2 (t,\cdot))(x\,)\,dt - \mu_2 R_2 (t,x)\,dt + rR_1 (t,x)\,dt + o(dt).
\end{align*}

From this we get as $dt\to0$:
\begin{equation}\label{example_second_eq}
\frac{\partial R_2 (t,x)}{\partial t} = (\mathcal L_2 R_2 (t,\cdot))(x) - \mu_2 R_2 (t,x) + rR_1 (t,x),\qquad
R_2 (0,x) = 0.
\end{equation}

Firstly, solve equation Equation~\eqref{example_first_eq}. Write again the equation for this function:
\[
\frac{\partial R_1 (t,x)}{\partial t} = (\beta - \mu_1-r) R_1 (t,x) + (\mathcal L_1 M_1 (t,\cdot))(x),\qquad
R_1 (0,x) = \delta_0 (x).
\]

To solve the equation apply the discrete Fourier transform Equation~\eqref{fourier_transform}. Then,
\[
\frac{\partial \widehat{R}_1 (t,\theta)}{\partial t} = (\beta - \mu_1 - r) \widehat{R}_1 (t,\theta) + \varkappa_1 \widehat{a}_1 (\theta) \widehat{R}_1 (t,\theta),\qquad
\widehat{R}_1 (0,x) = 1.
\]

The solution has the form:
\[
\widehat{R}_1 (t,\theta) = e^{(\beta-\mu_1 - r)t} e^{\varkappa_1 \widehat{a}_1 (\theta)t}
\]

\begin{remark}
For convenience, we denote the inverse Fourier transform~\eqref{inv_fourier_transform} of a function $f(\theta)$ by $\widetilde{\widehat{f}(\theta)}$.
\end{remark}

Therefore,
\begin{equation}\label{example_first_solution}
R_1 (t,x) = e^{(\beta-\mu_1 - r)t} \widetilde{e^{\varkappa_1 \widehat{a}_1 (\theta)t}}.
\end{equation}

We find out the solution of Equation~\eqref{example_second_eq}.
Substituting the solution for $R_1 (t,x)$ from Equation~\eqref{example_first_solution} in Equation~\eqref{example_second_eq}, we obtain:
\begin{align}\label{R_2(t,x)_diff_eq}
\frac{\partial R_2 (t,x)}{\partial t} = (\mathcal L_2 R_2 (t,\cdot))(x) - \mu_2 R_2 (t,x) + re^{(\beta-\mu_1 - r)t} \widetilde{e^{\varkappa_1 \widehat{a}_1 (\theta)t}}, \qquad R_2 (0,x) = 0.
\end{align}

To obtain the solution of differential Equation~\eqref{R_2(t,x)_diff_eq}, we consider two cases:
\[
\beta - \mu_1 - r = -\mu_2,\qquad \beta - \mu_1 - r \ne -\mu_2.
\]
\paragraph{Case $\beta - \mu_1 - r = -\mu_2$.}
Apply the discrete Fourier transform~\eqref{fourier_transform} and the variation of constants methods to solve the Equation~\eqref{R_2(t,x)_diff_eq}.

If $\varkappa_1 \widehat{a}_1 (\theta) = \varkappa_2 \widehat{a}_2 (\theta)$, then
\[
\widehat{R}_2 (t,\theta) = rte^{(\varkappa_2 \widehat{a}_2 (\theta) - \mu_2)t}.
\]

Consequently,
\[
R_2 (t,x) = rte^{-\mu_2 t} \widetilde{e^{\varkappa_2 \widehat{a}_2 (\theta)t}}.
\]

If $\varkappa_1 \widehat{a}_1 (\theta) - \varkappa_2 \widehat{a}_2 (\theta) =: d >0$, then
\[
\widehat{R}_2 (t,\theta) = \frac{r}{d} \Bigl(e^{dt} - 1\Bigr) e^{-\mu_2 t} e^{\varkappa_2 \widehat{a}_2 (\theta)t}.
\]

Then,
\[
R_2 (t,x) = re^{(\beta-\mu_1 - r)t}\widetilde{\Bigl(\frac{e^{\varkappa_1 \widehat{a}_1 (\theta)t}}{d}\Bigr)} - re^{-\mu_2 t}\widetilde{\Bigl(\frac{e^{\varkappa_2 \widehat{a}_2 (\theta)t}}{d}\Bigr)}.
\]

\paragraph{Case $\beta - \mu_1 - r \ne -\mu_2$.}
Here, the solution of differential Equation~\eqref{R_2(t,x)_diff_eq} has the form:

If $\varkappa_1 \widehat{a}_1 (\theta) + \beta - \mu_1 - r = \varkappa_2 \widehat{a}_2 (\theta) - \mu_2$, then
\[
\widehat{R}_2 (t,\theta) = rte^{-\mu_2 t} e^{\varkappa_2 \widehat{a}_2 (\theta)t};
\]

If $\varkappa_1 \widehat{a}_1 (\theta) + \beta - \mu_1 - r - \varkappa_2 \widehat{a}_2 (\theta) + \mu_2 =: d >0$, then
\[
\widehat{R}_2 (t,\theta) = \frac{r}{d} \Bigl(e^{dt} - 1\Bigr) e^{-\mu_2 t} e^{\varkappa_2 \widehat{a}_2 (\theta)t}.
\]

\begin{remark}\label{example_first_moments_f_v_j_asymptotics}
Find out the asymptotic behavior for the first moments $R_i (t,x)$, $i=1,2$ in a particular case. Assume that generators $\mathcal L_i$, $i=1,2$ defined in Equation~\eqref{generator} are equal, so that $\mathcal L_1 = \mathcal L_2$. Additionally, consider the case when underlying random walks have a finite variance of jumps, so that Equation~\eqref{f_v_j} is true. Then, from Equation~\eqref{asymptotic_f_v_j} we have, for each $x\in\mathbb{Z}^{d}$,
\begin{equation}\label{asymp_trans_prob}
\int_{[-\pi,\pi]^d} e^{\varkappa\widehat{a(\theta)}t} \cos((\theta,x))\,d\theta \sim \frac{\gamma_d}{t^{d/2}},
\end{equation}
where $\gamma_d$ is specified by Equation~\eqref{gamma_d}. With the usage of~\eqref{asymp_trans_prob} we get for $\widehat R_i (t,\theta)$, $i=1,2$ obtained above we have, for each $x\in\mathbb{Z}^{d}$, as $t\to\infty$:
\begin{align*}
&R_1 (t,x) \sim \frac{\gamma_d}{t^{d/2}},\\
&R_2 (t,x) \sim \frac{r\gamma_d}{t^{(d-2)/2}},
\end{align*}
when $\beta - \mu_1 - r = 0$ and $\mu_2 = 0$, and
\begin{align*}
&R_1 (t,x) \sim e^{At} \frac{\gamma_d}{t^{d/2}},\\
&R_2 (t,x) \sim \frac{r}{A} \Bigl( e^{At} - 1 \Bigr) \frac{\gamma_d}{t^{d/2}},
\end{align*}
when $A = \beta - \mu_1 - r\ne0$ and $\mu_2 = 0$
\end{remark}

So, we have found the first moments for both types of a particle. In the next sections, we are going to get the explicit form of the second moments.

\subsection{The Second Moment for $N_1 (t,x)$}
\label{example_second_moment_N_1(t,x)}

Here, we will find out the asymptotic behavior for the second moments of the random variable $N_1 (t,x)$.

To derive the second moment for $N_1 (t,x)$, we consider a more general problem. Let $N_1 (t,x,y)$ be the number of particles of the first type at time $t$ at point $y\in\mathbb Z^d$ generated by the single particle of the first type located at time $t=0$ at the site $x\in\mathbb Z^d$. The initial condition for $N_1 (t,x,y)$ is $N_1 (0,x,y) = \delta_x (y)$. If we use the designations from Section~\ref{example_description_of_the_model}, if $N_1 (t,x) = N_1 (t,0,x)$. Here, we will use the method of backward Kolmogorov equations. Define the generating function for the random variable $N_1 (t,x,y)$ as
\begin{equation}\label{gen_func}
F(t,x,y;z) = \Expect e^{-zN_1 (t,x,y)},
\end{equation}
where $z\in\mathbb R$, $t\geqslant0$. From this, we get the following lemma.

\begin{lemma}
The generating function $F(t,x,y;z)$ specified by Equation~\eqref{gen_func} satisfies the differential equation:
\begin{equation}\label{gen_func_diff_eq}
\begin{aligned}
\frac{\partial F(t,x,y;z)}{\partial t} &= (\mathcal L_{1,x} F(t,\cdot,y;z))(x) + f\bigl(F(t,x,y;z)\bigr) + r(1-F(t,x,y;z));\\
F(0,x,y;z) &= e^{-z\delta_x (y)}.
\end{aligned}
\end{equation}
\end{lemma}

\begin{proof}
Consider the generating function $F(t,x,y;z)$ at the time moment $t+dt$. Then,
\begin{align*}
F(t+dt,x,y;z) &= \Expect e^{-zN_1 (t+dt,x,y)}\\ &= \Expect\Bigl[ \Expect \Bigl[ e^{-zN_1 (dt,x,x)N_1 (t,x,y)} \prod_{u\ne0} e^{-zN_1(dt,x,x+u)N_1 (t,x+u,y)} | \mathcal F_{\leqslant t} \Bigr]\Bigr] \\&= \Expect e^{-zN_1 (t,x,y)} \biggl[ \sum_{n\geqslant2} e^{-z(n-1)N_1 (t,x,y)} b_n\, dt + e^{zN_1 (t,x,y)} (\mu_1 + r)\,dt \\&\quad+ \sum_{u\ne0} e^{-zN_1 (t,x+u,y)} \varkappa_1 a_1 (u)\, dt + 1 - \Bigl(\sum_{n\geqslant2} b_n + r + \mu_1 + \varkappa_1\Bigr)\,dt + o(dt) \biggr] \\&= F(t,x,y;z) + (\mathcal L_{1,x} F(t,\cdot,y;z))(x)\,dt + f\bigl(F(t,x,y;z)\bigr)\,dt \\&\quad+ r\bigl(1-F(t,x,y;z)\bigr)\,dt + o(dt),
\end{align*}
where $f(s) = \mu_1 + s(-\mu_1 - \sum_{n\geqslant2} b_n) + \sum_{n\geqslant2} b_n s^n$ and
\[
(\mathcal L_{i,x} \Psi (t,\cdot,y))(x) = \varkappa_i \sum_{v\ne0} a_i (v)[\Psi (t,x+v,y) - \Psi (t,x,y)],\quad i=1,2.
\]

Then,
\begin{align*}
F(t+dt,x,y;z) &- F(t,x,y;z) = (\mathcal L_{1,x} F(t,\cdot,y;z))(x)\,dt + f\bigl(F(t,x,y;z)\bigr)\,dt \\&\quad+ r\bigl(1-F(t,x,y;z)\bigr)\,dt + o(dt).
\end{align*}

Therefore, as $dt\to0$
\[
\frac{\partial F(t,x,y;z)}{\partial t} = (\mathcal L_{1,x} F(t,\cdot,y;z))(x) + f\bigl(F(t,x,y;z)\bigr) + r(1-F(t,x,y;z)).
\]

The initial condition for the latter equation follows from Equation~\eqref{gen_func}:
\[
F(0,x,y;z) = \Expect e^{-zN_1 (0,x,y)} = \Expect e^{-z\delta_x(y)} = e^{-z\delta_x(y)}.
\]
\end{proof}

Later, we also will use the following notation:
\[
(\mathcal L_{i,y} \Psi (t,\cdot,y))(x) = \varkappa_i \sum_{v\ne0} a_i (v)[\Psi (t,x,y+v) - \Psi (t,x,y)],\quad i=1,2.
\]

Let $M_1 (t,x,y) = \Expect N_1 (t,x,y)$ be the first moment of $N_1 (t,x,y)$.
Note that
\[
\frac{\partial M_1 (t,x,y)}{\partial t} = - \left. \frac{\partial^2 F(t,x,y;z)}{\partial t \partial z} \right|_{z=0}.
\]

Then, from Equation~\eqref{gen_func_diff_eq} we can derive the differential equation for the first moment taking the partial derivative of both sides of~\eqref{gen_func_diff_eq}. Omitting the calculus we obtain
\[
\frac{\partial M_1 (t,x,y)}{\partial t} = (\mathcal L_{1,x} M_1 (t,\cdot,y))(x) + (\beta-\mu_1 -r)M_1 (t,x,y),\qquad
M_1 (0,x,y) = \delta_x (y).
\]

As above, with the usage of the discrete Fourier transform~\eqref{fourier_transform}, we can find the solution for this equation:
\[
\widehat{M}_1 (t,\theta,y) = e^{i (\theta,y)} e^{(\beta-\mu_1 -r)t} e^{\varkappa_1 \widehat{a}_1 (\theta)t}.
\]

Then, using the inverse Fourier transform~\eqref{inv_fourier_transform}, we will obtain:
\begin{equation}
\label{eq: inter 1}
M_1 (t,x,y) = e^{(\beta-\mu_1 - r)t} \frac{1}{(2\pi)^d} \int_{[-\pi,\pi]^d} e^{\varkappa_1 \widehat{a}_1 (\theta)t} e^{i (\theta,y-x)}\,d\theta.
\end{equation}

\begin{remark}
Notice that from the obtained representation for $M_1 (t,x,y)$, we get that the first moment $M_1 (t,x,y)$ is a function that depends on the difference of the considered sites on the lattice, so that
\[
M_1 (t,x,y) = M_1 (t,0,y-x).
\]
\end{remark}

From Equation~\eqref{gen_func_diff_eq} (by taking partial derivative over parameter $z$ twice and substituting $z=0$) we can derive the differential equation for the second moment, which can be defined as $M_2 (t,x,y) = \Expect N_1^2 (t,x,y)$:
\[
\frac{\partial M_2 (t,x,y)}{\partial t} = \left. \frac{\partial^3 F(t,x,y;z)}{\partial t \partial z^2} \right|_{z=0}.
\]

Consequently, omitting the calculus
\begin{equation}
\label{sec_mom_f_t}
\begin{aligned}
\frac{\partial M_2 (t,x,y)}{\partial t} &= (\mathcal L_{1,x} M_2 (t,\cdot,y))(x) + (\beta-\mu_1 - r)M_2 (t,x,y) + \beta^{(2)} M_1^2 (t,x,y),\\
M_2 (0,x,y) &= \delta_x (y),
\end{aligned}
\end{equation}
where $\beta^{(2)} = \sum_{n\geqslant2} n(n-1)b_n$.

Apply the discrete Fourier transform~\eqref{fourier_transform} to this equation we obtain:
\begin{align*}
\frac{\partial \widehat{M}_2 (t,\theta,y)}{\partial t} &= (\varkappa_1 \widehat{a}_1 (\theta) + \beta-\mu_1 - r)\widehat{M}_2 (t,\theta,y) \\&\quad+ \frac{\beta^{(2)}}{(2\pi)^d} \int_{[-\pi,\pi]^d} \widehat{M}_1 (t,\theta-\psi,y) \widehat{M}_1 (t,\psi,y)\,d\psi;\\
\widehat{M}_2 (0,\theta,y) &= e^{i (\theta,y)}.
\end{align*}

\begin{remark}
As in Section~\ref{example_the_first_moments}, we consider the asymptotic behavior of $M_2 (t,x,y)$ in case when random walk for the particles of the first type has finite variance of jumps, so that
\[
\sum_{v} a_1 (v) |v|^2 <\infty.
\]
Consider Equation~\eqref{sec_mom_f_t}. The solution of this equation is the sum of the particular solution of Equation~\eqref{sec_mom_f_t} and the solution of homogeneous equation
\begin{equation*}
\frac{\partial M_2 (t,x,y)}{\partial t} = (\mathcal L_{1,x} M_2 (t,\cdot,y))(x) + (\beta-\mu_1 - r)M_2 (t,x,y) + \beta^{(2)} M_1^2 (t,x,y).
\end{equation*}
Let $M_{2,h} (t,x,y)$ be the solution of homogeneous equation and $M_{2,p} (t,x,y)$ be the particular one. Then, $M_2 (t,x,y) = M_{2,h} (t,x,y) + M_{2,p} (t,x,y)$.

Assume that $\beta-\mu_1 -r=0$. Then, previous equation takes the form
\begin{equation}\label{sec_mom_N_1(t,x,y)_remark}
\frac{\partial M_2 (t,x,y)}{\partial t} = (\mathcal L_{1,x} M_2 (t,\cdot,y))(x) + \beta^{(2)} M_1^2 (t,x,y).
\end{equation}

From Remark~\eqref{remark_p(t,x,y)} and Equation~\eqref{asymp_trans_prob}, we get that $M_{2,h} (t,x,y) \sim \frac{\gamma_d}{t^{d/2}}$, $t\to\infty$.
Note that from Remark~\ref{example_first_moments_f_v_j_asymptotics} we have that $M_1 (t,x,y)\sim\gamma_1/\sqrt{t}$ as $t\to\infty$ for $d=1$ and each $x,y\in\mathbb{Z}^{d}$. Let as $t\to\infty$
\[
f(t) = \beta^{(2)}\gamma_1^2\ln t + o(\ln t).
\]
Then, substituting $f(t)$ into Equation~\eqref{sec_mom_N_1(t,x,y)_remark} we have
\[
\frac{\beta^{(2)}\gamma_1}{t} + o(1/t) = \frac{\beta^{(2)}\gamma_1}{t} + o(1/t).
\]
Then, as $t\to\infty$, $f(t)$ is the solution of Equation~\eqref{sec_mom_N_1(t,x,y)_remark} and $M_{2,p} (t,x,y)\sim f(t)$, $t\to\infty$, for each $x,y\in\mathbb{Z}^{d}$.

Similarly, for $d\geqslant2$, we can find that, for each $x,y\in\mathbb{Z}^{d}$,
\[
M_{2,p} (t,x,y) \sim -\gamma_d^2 \frac{\beta^{(2)}}{(d - 1)t^{d-1}}
\]

Consequently, as $M_2 (t,x,y) = M_{2,h} (t,x,y) + M_{2,p} (t,x,y)$ we obtain that, for each ${x,y\in\mathbb{Z}^{d}}$,
\begin{alignat*}{2}
M_2 (t,x,y) &\sim \beta^{(2)}\gamma_1^2 \ln t&\quad\text{for}~d=1,\\
M_2 (t,x,y) &\sim \Bigl(\gamma_2 - \gamma_d^2 \beta^{(2)} \Bigr) t^{-1}&\quad\text{for}~d=2,\\
M_2 (t,x,y) &\sim \gamma_d {t^{-d/2}}&\quad\text{for}~d\geqslant3.
\end{alignat*}
\end{remark}

We have thus obtained the asymptotic behavior of the first two moments of the random variable $N_1 (t,x,y)$. In the next section, we will examine the effect of intermittency (see definition in the next section ) for the random variable $N_1 (t,x,y)$ using $M_i (t,x,y)$, $i=1,2$.

\subsection{Intermittency for $N_1 (t,x)$}
\label{example_intermittency_N_1(t,x)}

In Sections~\ref{example_the_first_moments} and~\ref{example_second_moment_N_1(t,x)}, we obtained the solutions for the first two moments for the random variable $N_1 (t,x,y)$. Here, we will study the effect of intermittency in the simplest case for the number of particles of the first type. Introduce the following definition (see, for example,~\cite{Getan}).
\begin{definition}
The field $\Lambda (t,x)$ is called intermittent when $t\to\infty$ if
\[
\lim_{t\to\infty}\frac{\Expect \Lambda^2 (t,x)}{\bigl(\Expect \Lambda(t,x) \bigr)^2} = \infty,
\]
where $x\in\Omega(t)$, and $\Omega(t)$ is a non-decreasing family of sets.
\end{definition}

\begin{remark}
We are going to consider the effect of intermittency for random variable $N_1 (t,x,y)$. In our designations $N_1 (t,y-x) = N_1 (t,0,y-x)$.
\end{remark}

In what follows, we are going to study the effect of intermittency in one area of $x$ and $y$  when $|y-x|=O(\sqrt{t})$ as $t\to\infty$.

Denote by $p(t,x,y)$ the solution of the following Cauchy problem
\[
\frac{\partial p(t,x,y)}{\partial t} = (\mathcal L_1 p(t,\cdot,y))(x),\qquad p(0,x,y) = \delta_x (y).
\]
Then, the representation from Equation~\eqref{eq: inter 1} has the form
\begin{equation}
\label{eq: inter 3}
M_1 (t, x, y) = p(t, x, y)e^{(\beta-\mu_1-r)t}.
\end{equation}

Using Duhamel's principle and Equation~\eqref{eq: inter 3}, from Equation~\eqref{sec_mom_f_t} we obtain
\begin{equation}
\label{eq: 4}
M_2(t, x, y) = M_1(t, x, y) + \beta^{(2)} \int_0^t \sum_{w \in \mathbb{Z}^d} M_1(t-s, x, w) M_1^2(s, w, y)\,ds.
\end{equation}

\begin{remark}
In cases when underlying random walk has infinite variance of jumps, so that relation (see definition in~\cite{Getan}(Equation~(1.2))) when $\nu := \beta-\mu_1-r > 0$, the results were obtained in~\cite{Getan}(Th.1.2).
\end{remark}
Now, consider the case where $\nu > 0$ and the underlying random walk has superexponentially light tails of a random walk such that for all $\lambda \in \mathbb{R}^{d}$
\[
\sum_{z \in \mathbb{Z}^d} e^{(\lambda,z)} a_1(z) < \infty.
\]

From~\cite{MolchanovYarovaya2013}(4.7) we have for $t \to \infty$ and $|x-y| = O(\sqrt{t})$
\[
p(t, x, y) = \frac{e^{-(B^{-1}(x-y), (x-y))/(2t)}}{(2\pi t)^{d/2}\sqrt{\det B}}+o(t^{-d/2}),
\]
where $B=(b^{(kj)})$ is the matrix with elements
\[
b^{(kj)} = \sum_{z \in \mathbb{Z}^d \setminus \{0\}} z_k z_j a_1(z),\quad k,j = 1,\ldots,d.
\]

Note that for all $x, y \in \mathbb{Z}^d$, $t > 0$
\begin{align*}
p(t, x, y) &= \int_{[-\pi, \pi]^d} e^{-i(\theta, y-x)} \widehat{p}(t, \theta, 0)\,d\theta = \int_{[-\pi, \pi]^d} \cos ( \theta, y-x ) \widehat{p}(t, \theta, 0)\,d\theta \\
&\leqslant \int_{[-\pi, \pi]^d} \widehat{p}(t, \theta, 0)\,d\theta = p(t, 0, 0).
\end{align*}

Then, for $t \to \infty$
\[
p(t, 0, 0) = \frac{C}{t^{d/2}} + o(t^{-d/2})
\]
and $0< p(t, 0, 0) < 1$, we have
\[
p(t, 0, 0) < \frac{C}{(t+1)^{d/2}}.
\]

For $t \to \infty$, $|x-y| = O(\sqrt{t})$ and $\nu > 0$ from Equation~\eqref{eq: inter 3} and \eqref{eq: 4} we obtain
\begin{equation*}
\begin{split}
\frac{M_2(t, x, y)}{M_1^2(t, x, y)} &= \frac{M_1(t, x, y)}{M_1^2(t, x, y)} + \frac{\beta^{(2)} e^{\nu t} \int_0^t e^{\nu s}\sum_{w \in \mathbb{Z}^d}p(t-s, x, w)p^2(s, w, y)\,ds}{e^{2\nu t} p^2(t, x, y)} \\
&< \frac{1}{e^{\nu t} p(t, x, y)} + \frac{\beta^{(2)} \int_0^t e^{\nu s} p(s, 0, 0) \sum_{w \in \mathbb{Z}^d}p(t-s, x, w)p(s, w, y)\,ds}{e^{\nu t} p^2(t, x, y)}
\end{split}
\end{equation*}

Using the Kolmogorov--Chapman equation, we obtain
\[
\sum_{w \in \mathbb{Z}^d}p(t-s, x, w)p(s, w, y) = p(t,x,y)
\]
in the numerator of the last summand. Then, we can continue the estimation
\begin{align*}
&= \frac{\beta^{(2)} p(t,x,y) \int_0^t e^{\nu s} p(s, 0, 0) ds}{e^{\nu t} p^2(t, x, y)} + o(1) \\&= \frac{\beta^{(2)} \int_0^t e^{\nu s} p(s, 0, 0) ds}{e^{\nu t} p(t, x, y)} + o(1) < C \int_0^t e^{\nu (s-t)} \bigg(\frac{t}{s+1}\bigg)^{d/2}\,ds + o(1)
\end{align*}

Thus, the random variable $N_1(t, x, y)$ is non-intermittent for $|x-y| = O(\sqrt{t})$.

\subsection{The Second Moment for $N_2 (t,x)$}
\label{example_second_moment_N_2(t,x)}
In this section, we will set up the differential equation for the second order correlation function for $N_2 (t,x)$ and determine the asymptotic behavior in the special case. Define the following correlation function for the second type particles
\[
R_{22} (t,x,y) = \Expect [N_2 (t,x) N_2 (t,y)].
\]

To obtain the differential equation for $R_{22} (t,x,y)$, we consider this random variable at time moment $t+dt$. Unlike the differential equation for $N_1 (t,x)$, here we consider two cases: $x=y$ and $x\ne y$.

Firstly, consider the case when $x=y$. Here, we have for $R_{22} (t+dt,x,x)$:
\begin{align*}
R_{22} (t+&dt,x,x)= \Expect N_2^2 (t+dt,x)\\& = \Expect N_2^2 (t,x) + 2\Expect N_2 (t,x) \Bigl[ \Expect[\psi(dt,x)|\mathcal F_{\leqslant t}] \Bigr] + \Expect\Bigl[ \Expect [\psi^2(dt,x)|\mathcal F_{\leqslant t}] \Bigr] \Bigr]\\& = R_{22} (t,x,x) + 2\Expect N_2 (t,x) \\&\quad\times \Bigl[ \varkappa_2 \sum_{z\ne0} a_2 (z) N_2 (t,x+z,x)\,dt + rN_1 (t,x)\,dt - \mu_2 N_2 (t,x)\,dt \\&\quad- \varkappa_2 N_2 (t,x)\,dt + o(dt) \Bigr] + \Expect \Bigl[ \varkappa_2 \sum_{z\ne0} a_2 (z) N_2 (t,x+z,x)\,dt + rN_1 (t,x)\,dt \\&\quad+ \mu_2 N_2 (t,x)\,dt + \varkappa_2 N_2 (t,x)\,dt + o(dt) \Bigr]\\ &= R_{22} (t,x,x) + 2(\mathcal L_{2,x} R_{22} (t,\cdot,x))(x)\,dt + 2rR_{12}(t,x,x)\,dt + (\mathcal L_2 R_2 (t,\cdot))(x)\,dt\\&\quad- 2\mu_2 R_{22} (t,x,x)\,dt + 2rR_1 (t,x)\,dt + \mu_2 R_2 (t,x)\,dt + 2\varkappa_2 R_2 (t,x)\,dt + o(dt).
\end{align*}

Let $R_{12}(t,x,y) = \Expect [N_1 (t,x)N_2 (t,y)]$. Then, as $dt\to0$ we obtain the differential equation for $R_{22} (t,x,x)$:
\begin{align*}
\frac{\partial R_{22} (t,x,x)}{\partial t} &= 2(\mathcal L_{2,x} R_{22} (t,\cdot,x))(x) - 2\mu_2 R_{22} (t,x,x) + 2rR_{12}(t,x,x) \\&\quad+ (\mathcal L_2 R_2 (t,\cdot))(x) + rR_1 (t,x) + \mu_2 R_2 (t,x) + 2\varkappa_2 R_2 (t,x);\\
R_{22} (0,x,x) &= 0.
\end{align*}

Now, consider the case when $x\ne y$. Then,
\begin{align*}
R_{22} (t,x,y) &= \Expect\Bigl[ \Expect[(N_2 (t,x) +\psi (dt,x))(N_2 (t,y) +\psi (dt,y))|\mathcal F_{\leqslant t}] \Bigr] = R_{22} (t,x,y) \\&\quad+ \Expect N_2 (t,x)\Bigl[\varkappa_2 \sum_{z\ne0} a_2 (z) N_2 (t,y+z) rN_1 (t,y)\,dt - \mu_2 N_2 (t,y)\,dt \\&\quad- \varkappa_2 N_2 (t,y)\,dt + o(dt)\Bigr] + \Expect N_2 (t,y)\Bigl[\varkappa_2 \sum_{z\ne0} a_2 (z) N_2 (t,x+z) rN_1 (t,x)\,dt \\&\quad- \mu_2 N_2 (t,x)\,dt - \varkappa_2 N_2 (t,x)\,dt + o(dt)\Bigr] - \Expect[\varkappa_2 a_2 (x-y) N_2 (t,x)\,dt \\&\quad+ \varkappa_2 a_2 (y-x) N_2 (t,y)\,dt + o(dt)] = R_{22} (t,x,y) + (\mathcal L_{2,x} R_{22} (t,\cdot,y))(x)\,dt \\&\quad+ (\mathcal L_{2,y} R_{22} (t,\cdot,y))(x)\,dt - 2\mu_2 R_{22} (t,x,y)\,dt + r\bigl(R_{12}(t,x,y) \\&\quad+ R_{12}(t,y,x)\bigr)\,dt - \varkappa_2 a_2 (x-y)\bigl(R_2 (t,x)+R_2 (t,y)\bigr)\,dt + o(dt).
\end{align*}

Therefore, as $dt\to0$ we have the differential equation for $R_{22} (t,x,y)$ when $x\ne y$:
\begin{align*}
\frac{\partial R_{22} (t,x,y)}{\partial t} &= (\mathcal L_{2,x} R_{22} (t,\cdot,y))(x) + (\mathcal L_{2,y} R_{22} (t,\cdot,y))(x) - 2\mu_2 R_{22} (t,x,x) \\&\quad- \varkappa_2 a_2 (x-y)\bigl(R_2 (t,x)+R_2 (t,y)\bigr) + r\bigl(R_{12}(t,x,y) + R_{12}(t,y,x)\bigr);\\
R_{22} (0,x,y) &= 0.
\end{align*}

Above, we have defined function $R_{12}(t,x,y)=\Expect N_1 (t,x)N_2(t,y)$. This is unknown function, consequently, we need to obtain the differential for this function. Using the same technique as for $R_{22}(t,x,y)$ we get:
\begin{align*}
R_{12}(t+dt,x,y) &= \Expect[\Expect [(N_1 (t,x)+\xi(dt,x))(N_2 (t,y)+\psi(dt,y))| \mathcal F_{\leqslant t} ]] \\&= R_{12}(t,x,y) + (\beta-\mu_1-r-\mu_2)R_{12}(t,x,y)\,dt + r\Expect[N_1 (t,x)N_1 (t,y)]\,dt \\&\quad- \delta_x (y) rR_1 (t,x)\,dt + \bigl((\mathcal L_{1,x} R_{12}(t,\cdot,y))(x) + (\mathcal L_{2,y} R_{12}(t,\cdot,y))(x)\bigr)\,dt\\&\quad + o(dt),
\end{align*}

Then, as $dt\to0$ the differential equation for $R_{12}(t,x,y)$ is
\begin{equation}\label{example_R_12(t,x,y)}
\left.\begin{aligned}
\frac{\partial R_{12} (t,x,y)}{\partial t} &= (\mathcal L_{1,x} R_{12}(t,\cdot,y))(x) + (\mathcal L_{2,y} R_{12}(t,\cdot,y))(x) \\&\quad + (\beta-\mu_1-r-\mu_2)R_{12}(t,x,y)+ r\Expect[N_1 (t,x)N_1 (t,y)] \\&\quad - \delta_x (y) rR_1 (t,x);\\
R_{12}(0,x,y) &= 0.
\end{aligned}\right\}
\end{equation}

Let $R_{11} (t,x,y) = \Expect[N_1 (t,x)N_1 (t,y)]$. To get the behavior of this function, we also need to have the differential equation for it. Then, consider $R_{11} (t+dt,x,y)$:
\begin{align*}
R_{11} (t+dt,x,y) &= R_{11} (t,x,y) + 2(\beta-\mu_1-r)R_{11} (t,x,y) \,dt \\&\quad + \bigl((\mathcal L_{1,x} R_{11} (t,\cdot,y))(x)+ (\mathcal L_{1,y} R_{11} (t,\cdot,y))(x)\bigr)\,dt \\&\quad + \delta_x (y) \bigl((\sum_{n\geqslant2} (n-1)^2 b_n+\mu_1+r + 2\varkappa_1)R_1 (t,x)+(\mathcal L_1 R_1 (t,\cdot))(x)\bigr)\,dt\\&\quad - \varkappa_1 a_1 (x-y)\bigl(R_1 (t,x)+ R_1 (t,y)\bigr)\,dt + o(dt).
\end{align*}

If $dt\to0$, we obtain $R_{11} (t,x,y)$:
\begin{equation}\label{example_R_11(t,x,y)}
\left.\begin{aligned}
\frac{\partial R_{11} (t,x,y)}{\partial t} &= (\mathcal L_{1,x} R_{11} (t,\cdot,y))(x) + (\mathcal L_{1,y} R_{11} (t,\cdot,y))(x)\\&\quad + 2(\beta-\mu_1-r)R_{11} (t,x,y) \\&\quad+ \delta_x (y) \bigl((\beta+\mu_1+r)R_1 (t,x)-(\mathcal L_1 R_1 (t,\cdot))(x)\bigr)\\&\quad- \varkappa_1 a_1 (x-y)\bigl(R_1 (t,x)+ R_1 (t,y)\bigr);\\
R_{11} (0,x,y) &= \delta_0 (x) \delta_0 (y).
\end{aligned}\right\}
\end{equation}

In the future calculus, we need the following lemma.
\begin{lemma}\label{lemma_G(t,x,y)_K(t,x,y)}
Let $G(t,x,y) := R_1 (t,x) R_2 (t,y)$ and $K(t,x,y) = R_1 (t,x) R_1 (t,y)$. Then, functions $G(t,x,y)$ and $K(t,x,y)$ satisfy the following differential equations
\begin{align*}
\frac{\partial G(t,x,y) }{\partial t} &= (\mathcal L_{1,x} G(t,\cdot,y))(x) + (\mathcal L_{2,y} G(t,\cdot,y))(x) + (\beta-\mu_1-r-\mu_2)G(t,x,y) \\&\quad+ rK(t,x,y);\\
G(0,x,y) &= 0.
\end{align*}
\begin{align*}
\frac{\partial K(t,x,y)}{\partial t} &= (\mathcal L_{1,x} K(t,\cdot,y))(x) + (\mathcal L_{1,y} K(t,\cdot,y))(x) + 2(\beta-\mu_1-r)K(t,x,y);\\
K(0,x,y) &= \delta_0 (x)\delta_0 (y).
\end{align*}
\end{lemma}

\begin{proof}
Note that
\[
\frac{G(t,x,y)}{\partial t} = R_2 (t,y) \frac{\partial R_1 (t,x)}{\partial t} + R_1 (t,x) \frac{\partial R_2 (t,y)}{\partial t}.
\]

Then, with the usage of Equations~\eqref{example_first_eq} and \eqref{example_second_eq}, we have
\begin{align*}
\frac{G(t,x,y)}{\partial t} &= R_2 (t,y) \frac{\partial R_1 (t,x)}{\partial t} + R_1 (t,x) \frac{\partial R_2 (t,y)}{\partial t} = R_2 (t,y) \bigl( (\beta - \mu_1 - r) R_1 (t,x) \\&\quad+ (\mathcal L_1 R_1 (t,\cdot))(x) \bigr) + R_1 (t,x) \bigl( (\mathcal L_2 R_2 (t,\cdot))(y) - \mu_2 R_2 (t,y) + rR_1 (t,y) \bigr) \\&= (\mathcal L_{1,x} G(t,\cdot,y))(x) + (\mathcal L_{2,y} G(t,\cdot,y))(x) + (\beta-\mu_1-r-\mu_2) G(t,x,y)\\&\quad + rK(t,x,y).
\end{align*}

Similarly, we can obtain the differential equation for $K(t,x,y)$. Notice that
\[
\frac{K(t,x,y)}{\partial t} = R_1 (t,y) \frac{\partial R_1 (t,x)}{\partial t} + R_1 (t,x) \frac{\partial R_1 (t,y)}{\partial t}.
\]

Consequently, using Equation~\eqref{example_first_eq}, we get
\begin{align*}
\frac{K(t,x,y)}{\partial t} &= R_1 (t,y) \frac{\partial R_1 (t,x)}{\partial t} + R_1 (t,x) \frac{\partial R_1 (t,y)}{\partial t} = R_1 (t,y) \bigl( (\beta - \mu_1 - r) R_1 (t,x) \\&\quad+ (\mathcal L_1 R_1 (t,\cdot))(x) \bigr) + R_1 (t,x) \bigl( (\beta - \mu_1 - r) R_1 (t,y) + (\mathcal L_1 R_1 (t,\cdot))(y) \bigr) \\&= (\mathcal L_{1,x} K(t,\cdot,y))(x) + (\mathcal L_{1,y} K(t,\cdot,y))(x) + 2(\beta-\mu_1-r)K(t,x,y).
\end{align*}

The initial conditions follow from
\[
G(0,x,y) = R_1(0,x)R_2(0,y)= 0,\quad K(0,x,y) = R_1(0,x)R_1(0,y)=\delta_0(x)\delta_0(y).
\]
\end{proof}

Find out the asymptotic behavior of the second moment $R_{22} (t,x,y)$ in the particular case where we consider $\mathbb Z^d$. Let the generators of the random walk be identical for both types of particles, so that $\mathcal L_1 = \mathcal L_2 =: \mathcal L$. Assume that the random walk with generator $\mathcal L$ has finite variance of jumps (see~\eqref{f_v_j}). For $A=\beta-\mu_1 - r > 0$ and $\mu_2 = 0$, it was found in Section~\ref{example_the_first_moments} that, as $t\to\infty$, for each $x\in\mathbb{Z}^{d}$,
\[
R_1 (t,x) \sim \frac{e^{At}}{t^{d/2}},\quad R_2 (t,x) \sim \frac{re^{At}}{At^{d/2}}.
\]
Note that from Lemma~\ref{lemma_G(t,x,y)_K(t,x,y)} and Equation~\eqref{example_R_11(t,x,y)}, we obtain
\begin{align*}
\frac{\partial (K(t,x,y) - R_{11} (t,x,y))}{\partial t} &= (\mathcal L_{1,x} (K(t,\cdot,y) - R_{11} (t,\cdot,y)))(x)\\&\quad + (\mathcal L_{1,y} (K(t,\cdot,y) - R_{11} (t,\cdot,y)))(x) \\&\quad+ 2(\beta-\mu_1-r)(K(t,x,y) - R_{11} (t,x,y)) + F(R_1(t,x)),
\end{align*}
where $F(R_1(t,x))$ is the function which linearly depends on $R_1 (t,x)$. Then, from the representation of $R_1 (t,x)$ as $t\to\infty$ we have, for each $x\in\mathbb{Z}^{d}$,
\[
F(R_1(t,x)) \sim C(y-x) \frac{e^At}{t^{d/2}},
\]
where $C(\cdot)$ function which can be obtained from Equation~\eqref{example_R_11(t,x,y)}. Then, as $t\to\infty$
\[
K(t,x,y) - R_{11} (t,x,y) = \frac{Ce^{At}}{t^{d/2}},\quad C-\text{constant}.
\]

The same technique helps us to find out that
\[
G(t,x,y) - R_{12} (t,x,y) = \frac{C'e^{At}}{t^{d/2}},\quad t\to\infty\quad C'-\text{constant}.
\]

Due to the homogeneity of space, we can consider the following values: $R_{22} (t,u) := R_{22} (t,0,y-x)$.

Using the above results, rewrite the equation for the second correlation function:
\begin{align*}
\frac{\partial R_{22} (t,u)}{\partial t} &= 2(\mathcal L R_{22} (t,\cdot))(u) + \frac{r^2 \gamma_1^2 (u) e^{2At}}{t^d} + \delta_0 (u) \frac{r\gamma_1 (u) e^{At}}{t^{d/2}} - a_2 (u) \frac{r\gamma_2 (u) e^{At}}{\sqrt{2\pi t}};\\
R_{22} (0,u) &= 0.
\end{align*}

Divide the last equation into two equations and find the solutions $R_{22}^{(1)} (t,u)$ and $R_{22}^{(2)} (t,u)$ of the following equations:
\begin{alignat*}{2}
\frac{\partial R_{22}^{(1)} (t,u)}{\partial t} &= 2(\mathcal L R_{22}^{(1)} (t,\cdot))(u) + \frac{r^2 \gamma_1 (u) e^{2At}}{t^d},&\qquad
R_{22}^{(1)} (0,u) &= 0;\\
\frac{\partial R_{22}^{(2)} (t,u)}{\partial t} &= 2(\mathcal L R_{22}^{(2)} (t,\cdot))(u) + \delta_0 (u) \frac{r e^{At}}{\sqrt{2\pi t}} - a_2 (u) \frac{r e^{At}}{t^{d/2}},&\qquad
R_{22}^{(2)} (0,u) &= 0.
\end{alignat*}

Then, we will have $R_{22} (t,u) = R_{22}^{(1)} (t,u) + R_{22}^{(2)} (t,u)$. The solution of the first equation for large $t$ is asymptotically as follows: $R_{22}^{(1)} (t,u) \sim \frac{C_1 (u)e^{2At}}{t^d}$ for each $u\in\mathbb{Z}^{d}$, whereas the solution of the second equation is asymptotically as follows: $R_{22}^{(2)} (t,u) \sim \frac{C_2 (u)e^{At}}{t^{d/2}}$ for each $u\in\mathbb{Z}^{d}$. Thus, $R_{22} (t,u) \sim \frac{C_1 (u)e^{2At}}{t^d}$ for each $u\in\mathbb{Z}^{d}$.

Here, for a fixed space coordinate, we do not have the intermittency effect:
\begin{align*}
\frac{R_{22} (t,x,x)}{R_2^2 (t,x)} \sim const <\infty,\quad t\to\infty,
\end{align*}
for each $x\in\mathbb{Z}^{d}$.

\section{Simulation of BRW}
\label{simulation_of_brw}

This section is devoted to process modeling using the Python programming language. We consider the state of the system as an array whose elements are lists of the form $[i, x, t_1, t_2]$, where $i$ characterizes the type of a particle and $x = (x_1, \dots,x_d)$ is its spatial coordinate, $t_1$ is the time of its entry into a given position (it was born at $x$ at this time or jumped at $x$ at this time), $t_2$ is the time of its exit from this position (it died at $x$ at this time or jumped out of $x$ at this time). Recall that we perceive all events related to the reproduction of offspring, including the degeneration from one type to another, as the death of the parent particle with the production of $k$ descendants of the first type and $l$ of the second type.

\textbf{Initialization.} We set the characteristics of BRW, $i = 1,2$:
\begin{enumerate}
\item $d$ is the lattice dimension;
\item $R$ is the array consisting of a finite number of lists $[i, x, 0, 0]$ characterizing types $i$ and particle positions $x = (x_1, \dots, x_d)$ at the initial moment of time;
\item $\varkappa_i$ are diffusion coefficients;
\item $A_i = (a_i (x,y))$ are matrices of the random walk intensities, by which the generators~\eqref{generator} are determined;
\item $\mu_i \geqslant 0$ are the death intensities;
\item $\beta_i(k, l) \geqslant 0$ are the birth intensities;
\item $r \geqslant 0$ is the intensity of degeneration from the first type to the second;
\item $T > 0$ is the duration of evolution under consideration.
\end{enumerate}

\textbf{Algorithm step.} We select one of the elements $[i, x, t_1, t_2]$ of the array $R$ such that $t_2 < T$. The particle spends exponential time $dt$ at the current position $x$, after which it does one of the following:
\begin{enumerate}
\item with probability $\mu_i/(\varkappa_i+\mu_i+\sum_{k+l\geqslant2} \beta_i(k, l) + r_i)$ dies;
\item with probability $\beta_i (k,l)/(\varkappa_i+\mu_i+\sum_{k+l\geqslant2} \beta_i(k, l) + r_i)$ divides into $k+l$ particles, then we append $k$ lists $[1, x, t_2, t_2+dt]$ and $l$ lists $[2, x, t_2, t_2+dt]$ to the array;
\item with probability $\varkappa_i a(z)/(\varkappa_i+\mu_i+\sum_{k+l\geqslant2} \beta_i(k, l) + r_i)$ jumps from position $x$ to position $x+z \ne x$, then we append $[i, x, t_2, t_2+dt]$ to the array;
\item with probability $r/(\varkappa_1+\mu_1+\sum_{k+l\geqslant2} \beta_1(k, l) + r)$ turns into a particle of the second type, then we append $[2, x, t_2, t_2+dt]$
to the array.
\end{enumerate}

Considered $[i, x, t_1, t_2]$ moves from array $R$ $(processing)$ to array $H$ $(history)$.

\textbf{Stop condition.} Algorithm steps are followed until there are elements $[i, x, t_1, t_2]$ in the array $R$ such that $t_2 < T$.

\textbf{Data analysis.} After the process is completed, the entire history of particles in different states is located in arrays $R$ and $H$. To find out the number, type and spatial configuration of particles at time $t$, we select those elements $[i, x, t_1, t_2]$ of the array $H$ for which $t_1 \leqslant t < t_2$.

\textbf{Simulations.} Suppose $d=1$ and at time $t=0$ there are $300$ initial particles on segment $[0, 300]\in\mathbb Z$. The random walk for the particles of the first type has intensities $a_1 (z) = 1/2$ for $|z| = 1$ and $\varkappa_1 = 1$, for the second type $a_2 (z) = 1/6$ for $|z|=1,2,3$ and $\varkappa_2 = 4$. Figure~\ref{g3} shows the simulation with parameters $\mu_1=0.25$, $\beta_1 (2,0) = 0.125$, $\beta_1 (1,1) = 0.125$, $\mu_2 = 0.375$, $\beta_2 (0,2)=0.125$, $\beta_2 (1,1) = 0.25$, all other birth/death intensities equal to $0$. It demonstrates the clustering effect in the case of the critical branching law described in~Section~\ref{Cluster}.

\begin{figure}[htbp!]
\centering
\includegraphics[width=0.8\textwidth]{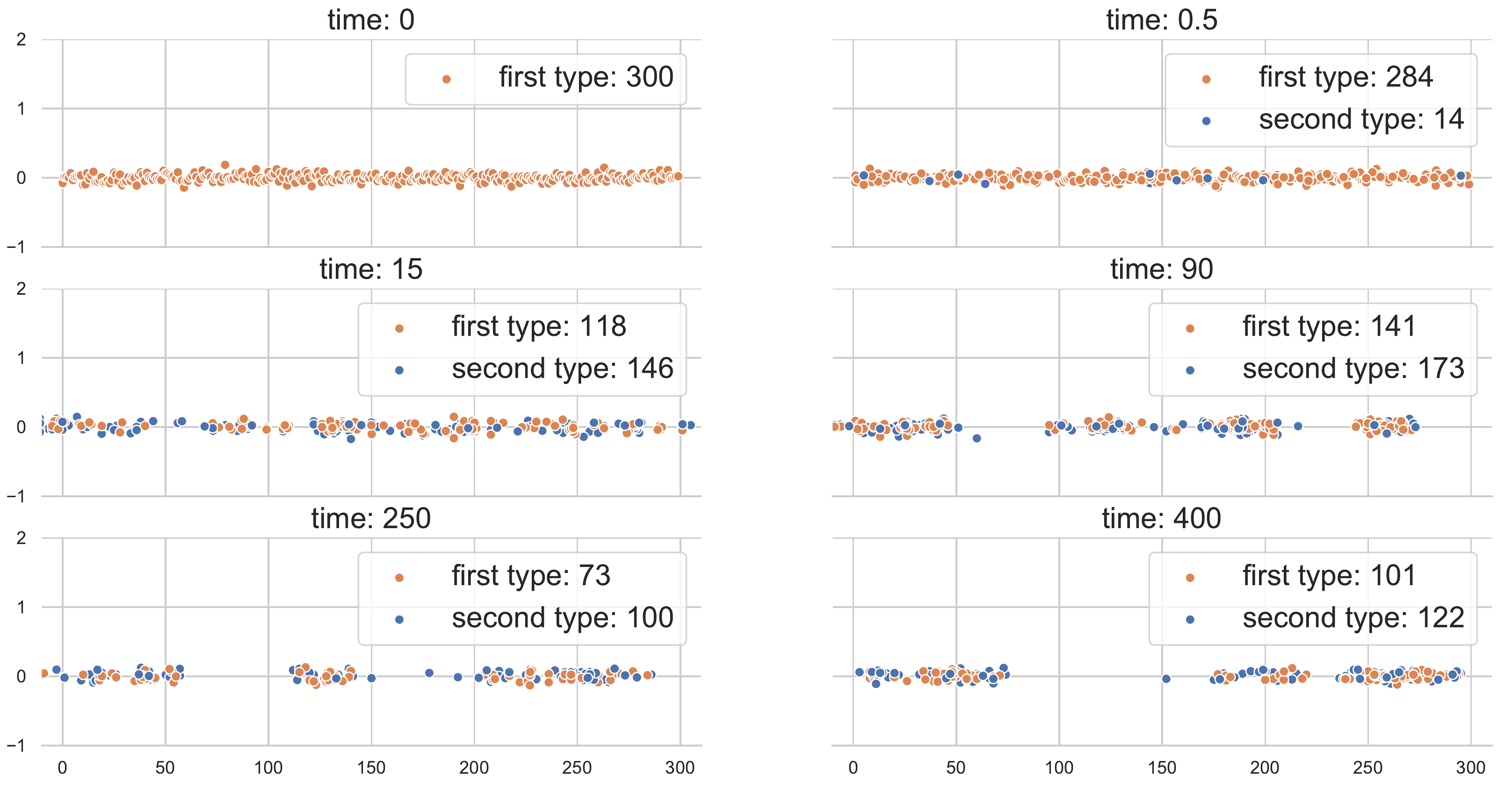}
\caption{Particle populations on $\mathbb Z^1$}
\label{g3}
\end{figure}

Suppose $d=2$, at the initial time $ t = 0$ there are $200$ particles of the first type at $(x_1, x_2) = (0, 0)$. We present the results for the case when the particles of the second type cannot produce offsprings (this was considered in the example), with the following parameters: The walk of the particles of the first type is set with a generator $\mathcal{L}_1$ with $\varkappa_1 = 1$ and intensities $a_1(z) = 1/80$, where $z=\{(z_1,z_2)\ne(0,0): z_1,z_2\in\mathbb Z, -4\leqslant z_1\leqslant4, -4\leqslant z_2\leqslant4\}$. The walk of particles of the second type is set using a generator $\mathcal{L}_2$ with $\varkappa_2 = 1$ and intensities $a_2(z) = 1/24$, where $z=\{(z_1,z_2)\ne(0,0): z_1,z_2\in\mathbb Z, -2\leqslant z_1\leqslant2, -2\leqslant z_2\leqslant2\}$. We record the number and arrangement of particles of both types at 6 time points. This simulation follows the model in Section~\ref{important_example}. Figure~\ref{g1} shows the simulation with parameters $\mu_1 = 0.05$, $\beta_1 (2,0) = 0.5$, $r = 0.45$, all other birth/death intensities equal to $0$.

\begin{figure}[htbp!]
\centering
\includegraphics[width=0.8\textwidth]{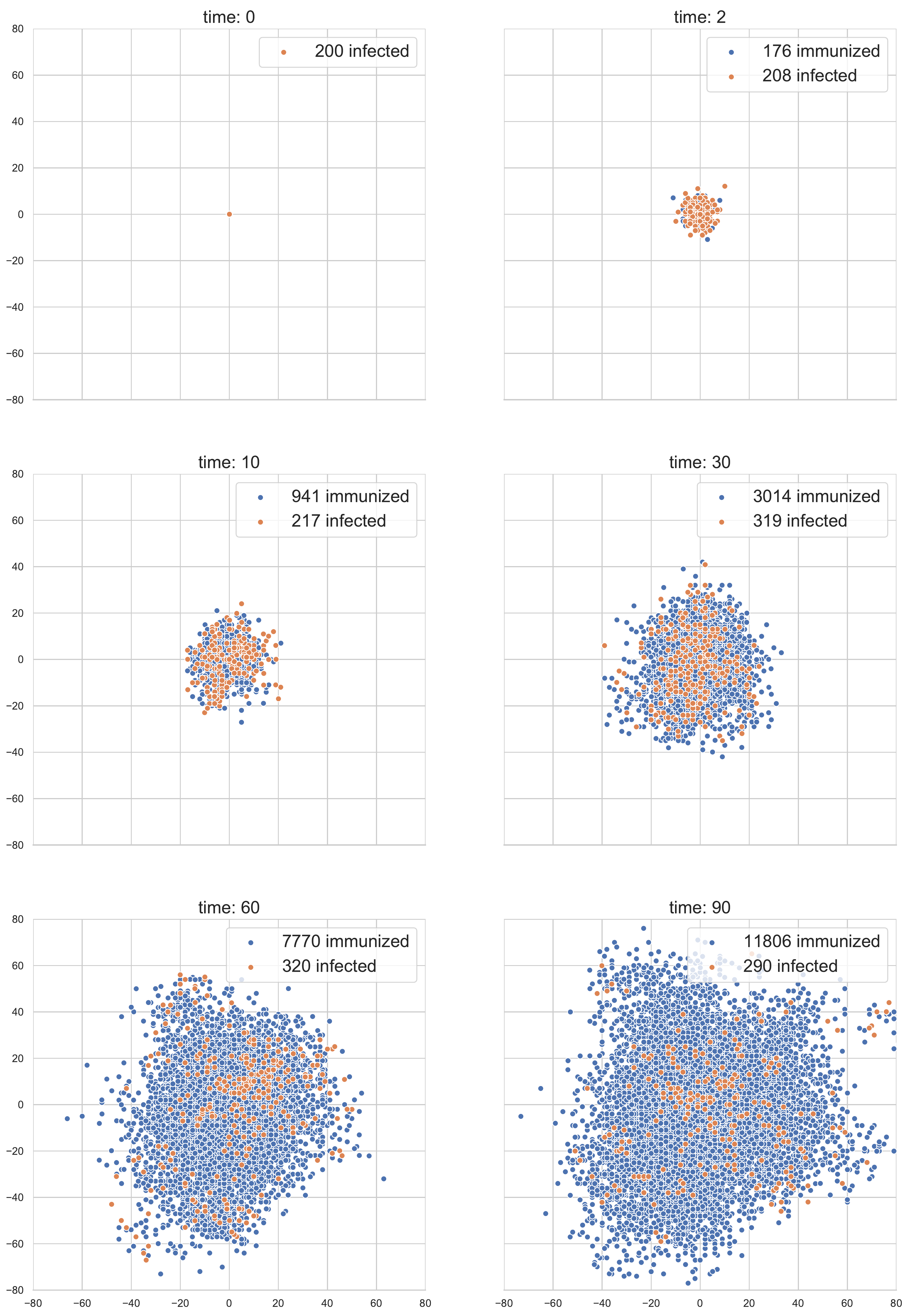}
\caption{Particle populations on $\mathbb Z^2$}
\label{g1}
\end{figure}

\section{Conclusions}

In this work, we considered a continuous-time branching random walk with two types of particles. The main results were devoted to the study of the limiting behavior of the moments of subpopulations generated by a single particle of each type. In particular, in Section~\ref{the_first_moments}, we have derived the differential equations for the first moments of subpopulations and found their solutions, which allows us to find exact expressions for their asymptotics. Similar results for the second moments were obtained in Section~\ref{the_second_moments}. In Section~\ref{Cluster}, we have shown that for particles of both types with the underlying recurrent random walks on $\mathbb Z^d$ a phenomenon of clustering of particles can be observed over long times, which means that the majority of particles are concentrated in some particular areas. The obtained results were then applied in Section~\ref{important_example} to study epidemic propagation. In this model, we considered two types of particles: infected and immunity generated. At the beginning, there is an infected particle that can infect others. Here, for the local number of particles of each type at a lattice point, we study the moments and their limiting behavior. Additionally, the effect of intermittency of the infected particles was studied for a supercritical branching process at each lattice point. To demonstrate the effect of limit clustering for the epidemiological model, we provided the results of a numerical simulation in Section~\ref{simulation_of_brw}.

We would like to emphasize that the present work is primarily devoted to a theoretical analysis of the effects that occur in branching random walks with two types of particles. We have tried to illustrate the obtained theoretical results by a numerical simulation. However, this simulation should not be considered as a full-fledged numerical analysis of the considered situation. Therefore, the question of developing real programs (in Python, R or any other language) for the numerical analysis of the considered processes arises quite naturally. In this paper, the authors did not undertake such a task, mainly because the fact that the computational aspects of modeling multidimensional processes are a special science that can hardly be treated professionally and completely in one or a few sections of even such an extensive article as ours. Possibly, further special studies will be devoted to it.

\section*{Acknowledgments} Iu.\,Makarova, D.\,Balashova and E.\,Yarovaya were supported by the Russian Foundation for the Basic Research (RFBR), project No. 20-01-00487; S.\,Molchanov was supported by the Russian Science Foundation (RSF), project No. 20-11-20119.


\end{document}